\documentclass[sn-mathphys-num]{sn-jnl}

\usepackage{graphicx}%
\usepackage{multirow}%
\usepackage{amsmath,amssymb,amsfonts}%
\usepackage{amsthm}%
\usepackage{mathrsfs}%
\usepackage[title]{appendix}%
\usepackage{xcolor}%
\usepackage{textcomp}%
\usepackage{manyfoot}%
\usepackage{booktabs}%
\usepackage{algorithm}%
\usepackage{algorithmicx}%
\usepackage{algpseudocode}%
\usepackage{listings}%

\usepackage{bbm}            			
\usepackage{mathalfa}
\usepackage{enumitem} 
\usepackage[T2A]{fontenc}

\theoremstyle{thmstyleone}%
\theoremstyle{definition}
\newtheorem{theorem}{Theorem}
\newtheorem{proposition}[theorem]{Proposition}%
\newtheorem{lemma}[theorem]{Lemma}
\newtheorem{corollary}[theorem]{Corollary}
\newtheorem{definition}[theorem]{Definition}
\newtheorem{example}[theorem]{Example}
\newtheorem{remark}[theorem]{Remark}

\begin{document}

\title[Article Title]{Image transformations, Markov operators, and sample median}

\author*[1]{\fnm{S. V.} \sur{Butler}}\email{svetbutler@gmail.com}

\affil*[1]{\orgdiv{Department of Mathematics}, \orgname{University of California, Santa Barbara}, 
\orgaddress{\street{South Hall, Room 6007}, \city{Santa Barbara}, \postcode{93106}, \state{CA}, \country{USA}}}

\abstract{(I.) We consider generalizations of an iterated function system and the associated Markov operators.
A Markov operator, defined on the space of (deficient) topological measures on a locally compact space,
is an infinite convex linear combination of adjoints of (d-) image transformations. Restricted to measures, 
this Markov-Feller operator has a nonlinear dual operator given by an infinite convex linear combination of (conic) quasi-homomorphisms.
If (d-) image transformations are contractions with respect to the Kantorovich-Rubinstein metric, 
a Markov operator has the unique invariant (deficient) topological measure. 
Taking a compact space, finitely many inverses of contractions as image transformations, and restricting the Markov operator to measures gives 
the classical result from the theory of fractals. 
There are various relations between Markov operator 
and the iterated function system where adjoints of (d-) image transformations are contractions on 
the compact metric space of $\{0,1\}$-valued (deficient) topological measures. For instance, the invariant (deficient) topological measure 
is the composition of the fixed point of the IFS and the basic (d-) image transformation. 
(II.) We define a generalized distribution of the sample median (g.d.s.m.) for continuous proper maps using an image transformation. 
We show that the g.d.s.m. and the inverse on the sample median are equivariant under solid variables, a large collection of transformations. On 
$\mathbb{R}^n$ such transformations include rotations, translations, symmetries, stretching, projections, monotone maps, etc.    
(III.) We show that a (signed) topological measure on a locally compact space with the covering dimension $\dim X \le 1$ is a (signed) Radon measure. }





\maketitle


\section{Introduction}

This paper covers three topics. (I.) We consider  generalizations of an iterated function system and the associated Markov operator. We    
discuss the unique fixed point and also obtain a Markov-Feller operator that has a nonlinear dual operator. 
(II.)  We define a generalized distribution of the sample median (g.d.s.m.) for continuous proper maps and show that the g.d.s.m. 
and the inverse on the sample median are equivariant under a wide collection of transformations. 
Although our generalizations belong to very different areas -- theory of fractals and probability $\&$ statistics --they use the same tools, 
namely, certain quasi-linear maps,  (d-) image transformations,  and their adjoints. 
(III.)  We show that a (signed) topological measure on a locally compact space with the covering dimension $\dim X \le 1$ is a (signed) Radon measure. 

This paper relies on results from \cite{Butler:QLM}, where we study (d-) image transformations and (conic) quasi-linear maps.
Quasi-linear maps are nonlinear operators from $C_0(X)$ (where $X$ is a locally compact space) 
to a normed linear space $E$ which are linear on the smallest subalgebras  generated by single functions. i.e. on 
$B(f) =  \{ \phi \circ f : \,  \phi \in C(\overline{f(X)}) \} $  (with $ \phi(0) = 0 $ if $X$ is noncompact) where $f \in C_0(X)$. 
($f \in C_0(X)$ if for every $\epsilon >0$ the set 
 $\{ x : | f | \ge \epsilon \}$ is compact.)
Conic quasi-linear maps are nonlinear operators which preserve nonnegative linear combinations on positive cones generated by single functions, 
i.e. on $ A^+(f) = \{ \phi \circ f: \ \phi \in C(\overline{f(X)}), \phi  \mbox{   is nondecreasing}\} $, (with $ \phi(0) = 0 $ 
if $X$ is noncompact)  where $f \in C_0(X)$.
When $E = \mathbb R$, we obtain quasi-linear functionals and p-conic quasi-linear functionals. 
Such nonlinear functionals (called quasi-integrals) correspond to topological measures and deficient topological measures, respectively.
These set functions generalize measures, and their definitions are close to that of a regular Borel measure, but they also have
some striking differences from Borel measures. For instance, (deficient) topological measures are not subadditive.  
The theory of quasi-linear functionals and topological measures, introduced by J. F. Aarnes \cite{Aarnes:TheFirstPaper},
has origins in mathematical axiomatization and interpretations of quantum physics. This theory has connections with different areas in mathematics, 
including probability and statistics, fractals, and Choquet integrals,
and has proved to be very influential in symplectic geometry, 
playing an important role in function theory on symplectic manifolds (\cite{PoltRosenBook}). 

Quasi-linear maps  from $C_0(X)$ to $C_0(Y)$ 
that behave like homomorphisms on singly generated subalgebras or cones  are called (conic) quasi-homomorphisms. 
They correspond to (d-) image transformations, one of the main objects of this paper.  
Image transformations and d-image transformations generalize the idea of  $u^{-1}$ in image measures $\mu \circ u^{-1}$. 
They move subsets of one space to subsets of another space so that (deficient) 
topological measures on one space produce (deficient) topological measures on another space.  
The adjoint of a (d-) image transformation is a continuous operator between spaces of (deficient) topological measures 
equipped with the weak topology, and it plays
an important role in this paper. We will give some background about (deficient) topological measures, quasi-integrals, 
(d-) image transformations and quasi-homomorphisms in the next section. 
More information about  these concepts and interconnections between them can be found in \cite{Butler:QLM}. 

We consider a generalization of an iterated function system and 
Markov operators involving (deficient) topological measures on a locally compact metric space.
A well-known result of Hutchinson gives the existence of the invariant measure for operator  
$M(\nu) = \sum_{i=1}^n \alpha_i  \  \nu \circ u_i^{-1}$ associated with the iterated function system with probabilities $(X, u_i, \alpha_i, i=1, \ldots, n)$,
see   \cite{Hutchinson}, \cite[Ch. IX]{Barnsley}.
Here $u_i$ are contractions on a complete metric space $X,   \alpha_i > 0,  \sum_{i=1}^n \alpha_i = 1$. 
Replacing finitely many $u_i^{-1}$ by infinitely many (d-) image transformations 
leads to a (d-) image transformation system, a generalization of an iterated function system. 
We consider the associated continuous Markov operator $S$ between spaces of (deficient) topological measures, 
where  $S(\mu) = \sum_{i=1}^\infty \alpha_i q_i^* (\mu)$, 
an infinite convex linear combination of adjoints of (d-) image transformations. 
Restricted to the space of measures, 
such an operator is a Markov-Feller operator with a nonlinear dual operator
$T = \sum_{i=1}^\infty \alpha_i \theta_i$
 given by an infinite convex linear combination of (conic) quasi-homomorphisms that correspond to (d-) image transformations.
We show that any uniformly bounded in variation family $\mathcal{M}$ of (deficient) topological measures is a complete metric space with respect to the 
Kantorovich-Rubinstein metric. We prove that the considered Markov operator on $\mathcal{M}$ 
has a unique invariant (deficient) topological measure if (d-) image transformations are contractions 
with respect to the Kantorovich-Rubinstein metric.
If we take a compact metric space as the underlying space $X$,  finitely many image transformations, 
each of which is the inverse of a contraction on $X$, and
restrict a Markov operator to measures on $X$, we obtain the classical result mentioned above.
We consider various relations between the Markov operator 
and the iterated function system (IFS) where adjoints of (d-) image transformations are contractions on 
the compact metric space of $\{0,1\}$-valued (deficient) topological measures. For example, the invariant (deficient) topological measure 
is the composition of the fixed point of the IFS and the basic (d-) image transformation. 
These and other results are in Section \ref{SectMarkov}. Together with new results, we also generalize and expand many results of
J. Aarnes, O. Johansen,  A. Rustad, and O. Pedersen  
about finitely many image transformations and topological measures on compact spaces 
(\cite{Aarnes:ITfirst}, \cite{AarnesJohansenRustad}, \cite{Pedersen}).

In Section \ref{SectDimSTM} we prove an important theorem linking dimension theory and topological measures. It says that  
a (signed) topological measure on a locally compact space with the covering dimension $\dim X \le 1$ is a (signed) Radon measure. Our result 
generalizes all previous results (\cite{Wheeler}, \cite{Grubb:SignedqmDimTheory}, \cite{Svistula:Signed}) 
about (signed) topological measures on compact spaces.

In Section \ref{SectSampleMed} we move to some concepts in probability and statistics. 
We define a generalized distribution of the sample median (g.d.s.m.) for continuous proper maps $T_i : Y \longrightarrow X$, 
where $X$ and $Y$ are locally compact spaces. 
For an odd number of maps the g.d.s.m. coincides with the probability distribution of the sample median, and for the even number $2n$ of maps 
 the g.d.s.m. is the average of probability distributions of the $n$th  and $(n+1)$th order statistics. 
We show that the g.d.s.m. and the inverse on the sample median are equivariant under a wide collection of transformations, so called solid variables. 
On $\mathbb{R}^n$ they include rotations, translations, symmetries, stretching, projections, monotone maps, etc.    
The results in this section are closely related to the ideas of Rustad in \cite{AlfMedian} (presented for compact metric spaces) 
and are obtained by using image transformations and their adjoints. We do not require spaces to be metric. 

\section{Preliminaries}

In this paper we consider Hausdorff spaces. We may abbreviate and write LCH for  "locally compact Hausdorff" and LC for  "locally compact". 
By $C(X)$ we denote the set of all real-valued continuous functions on $X$ with the extended uniform norm, 
by $C_b(X)$ the set of bounded continuous functions on $X$, and  
by $C_c(X)$ the set of continuous functions with compact support.  
$C_0^+(X)$, $C_c^+(X)$, $C^+(X)$ denote the collections of all nonnegative functions from $C_0(X)$, $C_c(X)$, $C(X)$, respectively.
When we consider maps into extended real numbers we assume that any such map is not identically $\infty$. 
We denote by $\overline E$ the closure of a set $E$, and by $ \bigsqcup$ a union of disjoint sets.
Notation $K_t \searrow K$ means that a decreasing net of sets ($t < s \Longrightarrow K_s \subseteq K_t $) decreases to 
$K = \bigcap_{t \in T} K_t$. Similarly, $U_t \nearrow U$ stands for an increasing to $U = \bigcup_{t \in T} U_t$ net of sets.
We denote by $id$ the identity function $id(x) = x$, 
and by $1_K$ the characteristic function of a set $K$. By $ supp \,  f $ we mean $ \overline{ \{x: f(x) \neq 0 \} }$.
By $ \delta_y$ we denote the point mass at $y$, and by $i_Y$ the map that assigns each $y \in Y$ the measure $ \delta_y$.
Several collections of sets are used often.   They include:
$\mathscr{O}(X)$;
$\mathscr{C}(X)$; and
$\mathscr{K}(X)$-- 
the collection of open subsets of $X$;  the collection of closed subsets of   $X $;
and the collection of compact subsets of   $X $, respectively.

\begin{definition} \label{MDe2}
Let $X$ be a  topological space and $\nu$ be a set function on a family $\mathcal{E}$ of subsets of $X$ that 
contains $\mathscr{O}(X) \cup \mathscr{C}(X)$
with values in $[0, \infty]$. 
We say that $\nu$ is simple if it only assumes  values $0$ and $1$; $\nu$ is finite if $ \nu(X) < \infty$.
A measure on $X$ is a countably additive set function on a $\sigma$-algebra of subsets of $X$ with values in $[0, \infty]$.  
A Borel measure on $X$ is a measure on the Borel $\sigma$-algebra on $X$.  
A Radon measure  $m$  on $X$ is a compact-finite Borel measure that is outer regular on all Borel sets, and inner compact regular on all open sets, i.e.
$m(K) < \infty$ for every compact $K$, 
$ m(E) = \inf \{ m(U): E \subseteq U, U \text{  is open} \} $ for every Borel set $E$, and 
$m(U) = \sup \{  m(K): K \subseteq U, K  \text{  is compact} \}$ for every open set $U$. 
A Borel measure is regular if on all Borel sets it is outer regular and inner compact regular.   
For a Borel measure  $m$ that is inner compact regular on all open sets (in particular, for a Radon measure)  we define $supp \ m$,  the support of  $m$, 
to be the complement of the largest open set $W$ such that $m(W) =0$. 
\end{definition}

\begin{definition}\label{TMLC}
A topological measure on $X$ is a set function
$\mu:  \mathscr{C}(X) \cup \mathscr{O}(X)  \longrightarrow  [0,\infty]$ satisfying the following conditions:
\begin{enumerate}[label=(TM\arabic*),ref=(TM\arabic*)]
\item \label{TM1} 
if $A,B, A \sqcup B \in \mathscr{K}(X) \cup \mathscr{O}(X) $ then
$
\mu(A\sqcup B)=\mu(A)+\mu(B);
$
\item \label{TM2}  
$
\mu(U)=\sup\{\mu(K):K \in \mathscr{K}(X), \  K \subseteq U\}
$ for $U\in\mathscr{O}(X)$;
\item \label{TM3}
$
\mu(F)=\inf\{\mu(U):U \in \mathscr{O}(X), \ F \subseteq U\}
$ for  $F \in \mathscr{C}(X)$.
\end{enumerate}
If in \ref{TM1}  the sets  $A,B$ are compact, then $\mu$ is called a deficient topological measure. 

For a (deficient) topological measure  $\mu$ we define $ \| \mu \| = \mu(X) = \sup \{ \mu(K): K \in \mathscr{K}(X) \}$.
We denote by $ \mathbf{DTM}(X)$ (respectively, $ \mathbf{TM}(X)$) the space of all finite deficient topological measures 
(resp., of all finite topological measures) on $X$.  $\mathbf{DTM}_1(X)$ (resp,.  $ \mathbf{TM}_1(X)$) stands for the set of 
deficient topological measures (resp, topological measures) satisfying $ \| \mu \| = 1$. 
\end{definition} 

\noindent
If two (deficient topological) measures agree on compact sets (or on open sets) then they coincide. If $\mu, \nu$ are deficient topological measures
and $\mu(A) \le \nu(A)$ for all open sets $A$ (or all compact sets $A$) then $ \mu \le \nu$.

The following theorem from \cite[Sect. 4]{Butler:DTMLC} gives criteria for a deficient topological measure to be a measure.

\begin{theorem} \label{subaddit}
Let $\mu$ be a deficient topological measure on a locally compact space $X$. 
The following are equivalent: 
\begin{itemize}
\item[(a)]
If $C, K$ are compact subsets of $X$, then $\mu(C \cup K ) \le \mu(C) + \mu(K)$.
\item[(b)]
If $U, V$ are open subsets of $X$,  then $\mu(U \cup V) \le \mu(U) + \mu(V)$.
\item[(c)]
$\mu$ admits a unique extension to an inner regular on open sets, outer regular Borel measure 
$m$ on the Borel $\sigma$-algebra of subsets of $X$. 
$m$ is a Radon measure iff $\mu$ is compact-finite. 
If $\mu$ is finite then $m$ is an outer regular and inner closed regular Borel measure.
\end{itemize}
\end{theorem}

\begin{remark} \label{tausm}
One may consult  \cite{Butler:DTMLC} for various properties of deficient topological measures on LC spaces.
A deficient topological measure is monotone, $\tau$-smooth on compact sets, and $\tau$-smooth on open sets, 
so, in particular, is additive on open sets.
If $ F \in \mathscr{C}(X)$ and $C \in \mathscr{K}(X)$ are disjoint, then $ \nu(F) + \nu(C) = \nu ( F \sqcup C)$.
A deficient topological measure $ \nu$ is also superadditive, i.e. 
if $ \bigsqcup_{t \in T} A_t \subseteq A, $  where $A_t, A \in \mathscr{O}(X) \cup \mathscr{C}(X)$,  
and at most one of the closed sets is not compact, then 
$\nu(A) \ge \sum_{t \in T } \nu(A_t)$. 
One consequence of superadditivity is this:
if $ \nu$ is a simple deficient topological measure on $X$ and $A$ is a closed or open set with $ \nu(A) =1$, then $ \nu(X \setminus A) = 0$.
If $ \nu(\{x\}) = 1$ for some $x$, then $ \nu(X \setminus \{x \}) = 0$, so $\nu$ is subadditive on open sets,  
and by Theorem \ref{subaddit} $\nu$ is the pointmass $ \delta_x$.
Unlike measures, simple (deficient) topological measures are not all the extreme (deficient) topological measures  
(see, for example, \cite{QfunctionsEtm}).
\end{remark}

\begin{remark} \label{Vloz}
Let $X$ be LC, and let $ \mathscr{M}$  be the collection of all Borel measures on $X$ that are inner regular on open sets and outer regular 
on Borel sets.  $  \mathscr{M}$ includes regular Borel measures and Radon measures. 
Denote by $M(X)$ the restrictions to $\mathscr{O}(X) \cup \mathscr{C}(X)$ of measures from $ \mathscr{M}$.
Then
 $M(X) \subsetneqq  TM(X)  \subsetneqq  DTM(X).$
The inclusions follow from the definitions. 
Information on proper inclusion
and various examples are in numerous papers, including
\cite{Aarnes:TheFirstPaper}, \cite{AarnesRustad},  \cite{QfunctionsEtm}, \cite{OrjanAlf:CostrPropQlf},   \cite {Svistula:Signed},  \cite{Svistula:DTM},
\cite{Butler:TechniqLC}, \cite{Butler:DTMLC}, and \cite{Butler:TMLCconstr}.
We also give some examples at the end of this section. 
\end{remark}

\begin{definition} \label{cqlf}
We call a  functional $\rho$ on $C_0(X)$  with values in $[ -\infty, \infty]$ (assuming at most one of $\infty, - \infty$) 
and $| \rho(0) | < \infty$ a p-conic quasi-linear functional if 
\begin{enumerate}[leftmargin=0.35in, label=(p\arabic*),ref=(p\arabic*)]
\item
$f\, g=0, \, f, g  \ge 0$ implies $ \rho(f+ g) = \rho(f) + \rho(g)$.
\item
 $0 \le g \le f$ implies $\rho(g) \le \rho(f)$.
\item
For each $f$, if $g,h \in A^+(f), \ a,b \ge 0$ then $\rho(a g + bh) = a \rho(g) + b \rho(h)$.
Here $ A^+(f) = \{ \phi \circ f: \ \phi \in C(\overline{f(X)}), \phi  \mbox{   is nondecreasing}\} $, (with $ \phi(0) = 0 $ 
if $X$ is noncompact) is a cone generated by $f$.
\end{enumerate}

\noindent
A map $\rho:C_0(X) \longrightarrow \mathbb{R}$
is a quasi-linear functional (or a positive quasi-linear functional) if 
\begin{enumerate}[leftmargin=0.25in, label=(\alph*),ref=(\alph*)]
\item \label{QIpositLC}
$ f \ge 0 \Longrightarrow \rho(f)  \ge 0.$
\item \label{QIlinLC}
For each  $f $,   if $g,h \in B(f),  \ a,b \in \mathbb{R}$ then $\rho(a g + bh) = a \rho(g) + b \rho(h)$.
Here  $B(f) =  \{ \phi \circ f : \,  \phi \in C(\overline{f(X)}) \} $  (with $ \phi(0) = 0 $ if $X$ is noncompact) is a subalgebra generated by $f$. 
\end{enumerate}
\end{definition}

\noindent
Note that $B(f)$ is the smallest closed subalgebra of $C_0(X)$ containing $f$ (resp, $f$ and constants if $X$ is compact), see \cite[L. 1.4]{Butler:QLFLC}. 

\noindent
For a functional $\rho$ on $C_0(X)$ we consider $\| \rho \| =  \sup \{ | \rho(f) | : \  \| f \| \le 1 \} $ and we say $\rho$ is bounded if $\| \rho \| < \infty$.

\begin{remark} \label{RemBRT}
Let $X$ be a LC space. 
There is an order-preserving bijection between finite deficient topological measures and bounded p-conic quasi-linear functionals.  
See \cite[Sect. 5,7,8]{Butler:ReprDTM}.
There is an order-preserving isomorphism between finite 
topological measures on $X$ and quasi-linear functionals 
of finite norm, and $\mu$ is a measure iff the 
corresponding functional is linear (see \cite[Th. 8.7]{Butler:ReprDTM}, \cite[Th. 3.9]{Alf:ReprTh}, \cite[Th. 15]{Svistula:DTM}, 
and \cite[Th. 3.9]{Butler:QLFLC}).
We outline the correspondence.
\begin{enumerate} [leftmargin=0.2in, label=(\Roman*),ref=(\Roman*)] 
\item \label{prt1}
Given a finite deficient  topological measure $\mu$ on $X$ and $f \in C_b(X)$, define functions on $\mathbb{R}$:
$$ R_1 (t) = R_{1, \mu, f} (t) =  \mu(f^{-1} ((t, \infty) )), \ \ \ \ \   R_2 (t) =  R_{2,  \mu, f} (t) =\mu(f^{-1} ([t, \infty) )). $$
Let $r=r_{f, \mu}$ be the Lebesque-Stieltjes measure associated with $-R_1$, a regular Borel measure on $ \mathbb{R}$. 
The $ supp \ r \subseteq \overline{f(X)}$.
We define a functional on $C_b(X)$ (in particular, on  $C_0(X)$):
\begin{align*} 
\mathcal{R} (f) & = \int _{\mathbb{R}}  id \, dr = \int_{[a,b]} id \, dr  =   \int_a^b R_1 (t) dt + a \mu(X)  =  \int_a^b R_2 (t) dt + a \mu(X), 
\end{align*}
where $[a,b]$ is any interval containing $f(X)$.
If $f(X) \subseteq [0,b]$ we have:
\begin{align*} 
 \mathcal{R} (f) = \int_{[0,b]}  id \, dr  =   \int_0^b R_1 (t) dt =   \int_0^b R_2 (t) dt.
\end{align*}
We say $\mathcal{R}$  is a quasi-integral (with respect to $ \mu$) and write:
\begin{align*} 
\int_X f \, d\mu = \mathcal{R}(f) = \mathcal{R}_{\mu} (f) =  \int _{\mathbb{R}}  id \, dr.
\end{align*}
If $\mu$ is a topological measure we may write $m = m_{f, \mu}$ instead of $r$.
\item  \label{RHOsvva}
Functional $\mathcal{R} $ is nonlinear. 
By \cite[L. 7.7,  Th. 7.10, L. 3.6, L. 7.12, Th. 8.7]{Butler:ReprDTM}  and  \cite[Cor. 4.10]{Butler:QLFLC} we have:
\begin{enumerate}[leftmargin=*]
\item
$\mathcal{R} (f) $ is positive-homogeneous, i.e. $\mathcal{R} (cf)  = c \mathcal{R} (f) $ for $c \ge 0$, $ f \in C_b(X)$. 
\item
$\mathcal{R} (0) =0$. 
\item 
$\mathcal{R}$ is monotone, i.e. if $ f \le g$  then $\mathcal{R} (f) \le \mathcal{R} (g)$ for $f, g \in C_b(X)$.
\item
$ \mu(X)  \cdot \inf_{x \in X} f(x)  \le \mathcal{R}(f)  \le \mu(X) \cdot \sup _{x \in X} f(x) $ for $f \in C_b(X)$. 
In particular, 
\begin{align} \label{estin}
  | \int_X f \, d\mu | \le \|f \| \mu(X).
\end{align}
\item
(orthogonal additivity) If $f g = 0$, where  $f, g \in C_b^+(X)$,  then $\mathcal{R} (f+g) = \mathcal{R} (f) + \mathcal{R} (g)$; \\
if $f g = 0 $, where $f \ge 0, g \le 0$, $f, g \in C_0(X)$, then $\mathcal{R} (f+g) = \mathcal{R} (f) + \mathcal{R} (g)$.
\item
(Lipschitz property)
if $f, g \in C_c^+(X), \, supp \, f, supp \, g \subseteq K$, $K$ is compact, then 
\begin{align} \label{LipCc} 
 | \mathcal{R}(f) - \mathcal{R}(g) |  = |  \int_X f \, d\mu - \int_X g \, d\mu | \le \| f - g \| \, \mu(K) \le \| f - g \| \, \mu(X). 
\end{align} 
\item 
If $ \mu$ is a finite topological measure,  $f, g \in C_0(X)$ then
\begin{align*} 
 \int_X f \, d\mu = \int_X f^+ \, d\mu - \int_X f^- \, d\mu, \ \  \ \ \ \ \ 
|\int_X f \, d\mu - \int_X g \, d\mu | \le  2 \| f - g \| \, \mu(X).
\end{align*}

\end{enumerate}
If $X$ is compact and  $\mu$ is a deficient topological measure, then 
$ |  \int_X f \, d\mu - \int_X g \, d\mu | \le \mu(X)  \| f - g \|. $
If $ \mu$ is a topological measure, the functional $ \mathcal{R}$ is a quasi-linear functional. 
If $ \mu$ is a deficient topological measure, the functional $ \mathcal{R}$ is a p-conic quasi-linear functional (which is enough to consider on $C_0^+(X)$).

\item \label{mrDTM} 
A functional $\rho$ with values in $[ -\infty, \infty]$ (assuming at most one of $\infty, - \infty$) and $| \rho(0) | < \infty$ 
is called a d-functional if   
on nonnegative functions it is positive-homogeneous, monotone, and orthogonally additive, i.e. for $f, g \in D(\rho)$ (the domain of $ \rho$) we have: 
(d1) $f \ge 0, \ a > 0  \Longrightarrow  \rho (a f) = a \rho(f)$; 
(d2) $0 \le  g \le f \Longrightarrow  \rho(g) \le \rho(f) $;
(d3) $f \cdot g = 0, f,g \ge 0  \Longrightarrow  \rho(f + g) = \rho(f) + \rho(g)$. 

Let  $\rho$ be a d-functional with $  C_c^+(X) \subseteq D(\rho) \subseteq C_b(X)$. 
In particular, we may take a bounded (p-conic) quasi-linear functional $ \mathcal{R}$ on $  C_0^+(X)$. 
The corresponding
deficient topological measure $ \mu = \mu_{\rho}$ is given as follows: 

If $U$ is open, $ \mu_{\rho}(U) = \sup\{ \rho(f): \  f \in C_c(X), 0\le f \le 1,  supp \, f\subseteq U  \}$,

if $F$ is closed, $ \mu_{\rho}(F) = \inf \{ \mu_{\rho}(U): \  F \subseteq U,  U \in \mathscr{O}(X) \}$, 

if $K$ is compact, $ \mu_{\rho}(K) = \inf \{ \rho(g): \   g \in C_c(X), g \ge 1_K \}  
= \inf \{ \rho(g): \   g \in C_c(X), 1_K \le g \le 1 \}. $
\end{enumerate}
If $\| \rho\| < \infty$ then the corresponding deficient topological measures is finite.  
If given a finite deficient topological measure $\mu$, we obtain $ \mathcal R$, and then $\mu_{ \mathcal R}$, then $ \mu = \mu_{ \mathcal R}$.
A bounded (p-conic) quasi-linear functional $ \mathcal{R}$  has the form given in part \ref{prt1} with $ \mu = \mu_{\mathcal{R}}$. 

For finite deficient topological measures $ \mu, \nu$
\begin{align} \label{EqByInt}
\int f \, d\mu = \int f \, d\nu  \mbox{    for all    }   f \in C_c^+(X)  \ \ \  \Longrightarrow \ \ \  \mu = \nu.
\end{align}
\end{remark}

We would like to give some examples. 

\begin{definition}
A set $A \subseteq X$ is called precomact if $\overline A$ is compact. 
If $X$ is LC, noncompact, a set $A$ is solid if $A$ is  connected, and none of the connected components of  $X \setminus A$ is precompact. 
If $X$ is compact, a set $A$ is solid if $A$ and $X \setminus A$ are connected.
\end{definition}

Typically, in papers on (deficient) topological measures and quasi-integrals, the term "bounded" is used instead of "precompact".
Here we use the word "precompact" to avoid confusion with bounded sets in metric spaces.

Let  $\mathscr{K}_{s}(X)$ and $\mathscr{O}_s^*(X)$ denote the collections of compact solid sets and precompact open solid sets in $X$, respectively.  
Let  $ \mathscr{A}_{s}^{*}(X) = \mathscr{K}_{s}(X) \cup \mathscr{O}_s^*(X)$.
 
\begin{definition} \label{DeSSFLC}
A function $ \lambda: \mathscr{A}_{s}^{*}(X) \rightarrow [0, \infty) $ is a solid-set function on $X$ if
\begin{enumerate}[label=(s\arabic*),ref=(s\arabic*)]
\item \label{superadd}
$ \sum_{i=1}^n \lambda(C_i) \le \lambda(C)$ whenever $\bigsqcup_{i=1}^n C_i \subseteq C,  \ \  C, C_i \in \mathscr{K}_{s}(X)$; 
\item \label{regul}
$ \lambda(U) = \sup \{ \lambda(K): \ K \subseteq U , \ K \in \mathscr{K}_{s}(X) \}$ for $U \in \mathscr{O}_{s}^{*}(X)$; 
\item \label{regulo}
$ \lambda(K) = \inf \{ \lambda(U) : \  K \subseteq U, \ U \in \mathscr{O}_{s}^{*}(X) \}$ for $ K  \in \mathscr{K}_{s}(X)$; 
\item  \label{solidparti}
$ \lambda(A) = \sum_{i=1}^n \lambda (A_i)$ whenever $A = \bigsqcup_{i=1}^n A_i, \ \ A , A_i  \in \mathscr{A}_{s}^{*}(X)$.
\end{enumerate}
\end{definition}

\begin{remark}
Many examples of topological measures that are not measures are obtained in the following way. Define a solid-set function on 
a LCH connected locally connected space. A solid set function extends to a unique topological measure. 
See \cite[Def. 2.3, Th. 5.1]{Aarnes:ConstructionPaper}, \cite[Def. 6.1, Th. 10.7]{Butler:TMLCconstr}.
This method is very convenient when $X$ is a compact Hausdorff connected locally connected space with genus 0 
(i.e. $X$ can not be a disjoint union of more than two nonempty solid sets), and such a space is called a q-space.
For a q-space, in part \ref{solidparti} of Definition \ref{DeSSFLC}  (typically, the hardest to verify) one only needs to check that
$\lambda(X) = \lambda(A) + \lambda(X \setminus A)$ for a solid set $A$. 
For a noncompact LCH connected locally connected space whose one-point compactification has genus 0 
(we shall call such a space a noncompact q-space)
this method is even simpler, for  part \ref{solidparti} of Definition \ref{DeSSFLC} holds automatically by \cite[L. 15.2]{Butler:TMLCconstr}.
For more information about solid sets, solid-set functions, equivalent definitions of a solid-set function and genus
see \cite[Rem. 6.3, Sect. 11, Sect. 12, L. 15.2]{Butler:TMLCconstr}.
\end{remark}

\begin{example} \label{ExDan2pt}
Suppose that  $ \lambda$ is the Lebesgue measure on $X = \mathbb{R}^2$,  and the set $P$ consists of points
$p_1 = (0,0)$ and $p_2 = (2,0)$.
For each precompact open solid or compact solid set $A$ let $ \nu(A) = 0$ if $A \cap P = \emptyset$,   
$ \nu(A) = \lambda(A) $ if $A$ contains one point from $P$, and 
$ \nu(A) = 2 \lambda(X)$ if $A$ contains both points from $P$.
Then $\nu$ is a solid-set function (see \cite[Ex. 15.5]{Butler:TMLCconstr}), and  $\nu$ extends to a unique topological measure on $X$. 
Let $K_i$ be the closed ball of radius $1$ centered at $p_i$ for $i=1,2$. Then 
$K_1, K_2$ and $ C= K_1 \cup K_2$ are compact solid sets, $\nu(K_1) = \nu(K_2) = \pi, \,  \nu(C) = 4 \pi$. Since 
$\nu$ is not subadditive, it is not a measure.  The quasi-linear functional corresponding to $ \nu$ is not linear. 
\end{example}

\begin{example} \label{nvssf}
Let  $X = \mathbb{R}^2$ or a square, $n$ be a natural number, and let $P$ be a set of distinct $2n+1$ points.
For each $A  \in \mathscr{A}_{s}^{*}(X)$ let $ \nu(A) = i/n$ if $ A$ contains  $2i$ or $2i+1$ points from $P$.
Then $ \nu$ is a solid-set function, and it extends to a unique topological measure on $X$ 
that assumes values $0, 1/n, \ldots, 1$. 
See  \cite[Ex. 2.1]{Aarnes:Pure},  \cite[Ex. 2.5]{AarnesRustad}, \cite[Ex. 4.14, 4.15]{QfunctionsEtm}, 
and \cite[Ex. 15.9]{Butler:TMLCconstr}.
This topological measure is not subadditive. 
(When $X$ is the square and $n=3$, it is easy to represent $X = A_1 \cup A_2 \cup A_3$,
where each $A_i$  is a closed solid set containing one point from $P$. Then $\nu(A_i) =0$ for $i=1,2,3$, while $\nu(X) = 1$.) 
Since $\nu$ is not subbadditive, it is not a measure, and the corresponding quasi-linear functional  $\rho$ is not linear. 
If we take $P =\{p, p, t \} $ or $ \{ p, p, p \}$ then the resulting topological measure is  
the point mass at $p$.
In \cite[Ex. 4.13]{Butler:QLFLC} for $n=5$ we show that there are $f,g \ge 0$ such that $ \rho(f+g) \neq \rho(f) + \rho(g)$. 
If $X$ is LC, noncompact, for the functional $\rho$ we consider
a new functional $ \rho_g$ defined by $\rho_g(f) = \rho(gf)$, where $g \ge 0$. 
The functional $\rho_g$ corresponds to a deficient topological measure 
obtained by integrating $g$ over closed and open sets with respect to $\nu$. We can choose $ g \ge 0$ or $ g >0$ so that  
$\rho_g$ is no longer linear on singly generated subalgebras, but only linear on singly generated cones. 
See \cite[Ex. 35, Th. 43]{Butler:Integration} for details.
\end{example}

\begin{example}[Aarnes circle topological measure] \label{Aatm}
Let $X$ be the unit disk in $ \mathbb{R}^2$ and $B$ be the boundary of $X.$
Fix a point $p$ in the interior of the circle.
Define $\mu $ on solid sets as follows:
$\mu (A) = 1$ if i) $B \subset A$ or
ii) $ p \in A $ and $A \cap B \ne  \emptyset$.
Otherwise $ \mu(A) = 0 $.
Then $ \mu $ is a solid-set function, and it extends to a simple topological measure on $X$.
Let $A_1$ be a closed solid set which is an arc that is a proper subset of $B$, 
$A_2$ be a closed solid set that is the closure of $B \setminus A_1$, 
and $A_3 = X \setminus B$ be
an open solid subset of $X$. Then 
$X =  A_1 \cup A_2 \cup A_3, \  \mu(X) = 1 $, but 
$ \  \mu (A_1) + \mu(A_2) + \mu(A_3) = 0$.
Since $ \mu$ is not subadditive,  it is not a measure, i.e. $\mu$ is simple, but not a point mass. 
When $X = \mathbb{R}^2$, given a closed ball $B(p, \epsilon)$ we may define a set function $\mu_{p, \epsilon}$ 
on precompact open solid and compact solid sets in 
the same way as above. Using \cite[Def. 6.1, Th. 3.10]{Butler:TMLCconstr} it is easy to see that  $\mu_{p, \epsilon}$ is a solid-set function, 
so it gives a topological 
measure on $\mathbb{R}^2$. 
\end{example}

\begin{example} \label{basicExDTM}
Let $X$ be locally compact, and let $D$ be a connected compact subset of $X$. Define a set function 
$\nu$ on $\mathscr{O}(X) \cup \mathscr{C}(X)$  by setting $\nu(A) = 1$ if $ D \subseteq A$ and $\nu(A) = 0$ otherwise.
If $D$ has more than one element, then $\nu$ is a deficient topological measure, but not 
a topological measure. See \cite[Ex. 6.1]{Butler:DTMLC} and \cite[Ex. 1, p.729]{Svistula:DTM} for details.
\end{example}

\noindent
More examples of topological measures and quasi-integrals on locally compact spaces
are in \cite{Butler:TechniqLC}, \cite{Butler:QLFLC}, and \cite[Sect. 15]{Butler:TMLCconstr};
more examples of deficient topological measures are in  \cite{Butler:DTMLC} and \cite{Svistula:DTM}. 

\begin{definition} \label{defwk}
The weak topology on $ \mathbf{DTM}(X)$ is the coarsest (weakest) topology for which maps 
$ \mu \longmapsto \mathcal{R}_{\mu} (f), f \in C_0^+(X) $ are continuous.
\end{definition}

\begin{remark} \label{rmwkbase}
The basic neighborhoods for the weak topology have the form 
$$N(\nu, f_1, \ldots, f_n, \epsilon) = \{ \mu \in \mathbf{DTM}(X): \ |\mathcal{R}_{\mu}(f_i) - \mathcal{R}_{\nu} (f_i) | < \epsilon, \, f_i \in C_0(X), \\
i=1, \ldots, n  \notag \}. $$
In the above line we may also take $f_i \in C_0^+(X)$ (see \cite[Sect. 2]{Butler:QLM}).  
Other ways to describe the weak topology are in \cite[Sect. 2]{Butler:WkConv},  \cite[Sect. 3]{Butler:QLM}.

Our definition of weak convergence corresponds to one used in probability theory. It is the same as a functional analytical definition of  $wk^*$ convergence
on $\mathbf{DTM}(X)$ (respectively, on $\mathbf{TM}(X)$), which is justified by the fact that this topology agrees with the weak$^*$ topology 
induced by p-conic quasi-linear functionals (resp., quasi-linear functionals). 
In many papers the term "$wk^*$-topology" is used. 
\end{remark}

\begin{remark}
By  \cite[Th. 8.7]{Butler:ReprDTM},  $ \mathbf{DTM}(X)$ (respectively, $ \mathbf{TM}(X)$)  is homeomorphic 
to the space of bounded p-conic quasi-linear (resp.,  bounded quasi-linear) functionals  endowed with pointwise convergence. 
$ \mathbf{DTM}(X)$ is a topological vector space, and  by Theorem \cite[Th. 2.4]{Butler:WkConv}  $ \mathbf{DTM}(X)$ is Hausdorff.
\end{remark}

\begin{definition} \label{unifbdv}
Let $X$ be locally compact.
A family $\mathcal{M}  \subseteq \mathbf{DTM}(X)$ is uniformly bounded in variation if there is a positive constant $M$ such that $ \| \mu \| \le M$ 
for each $\mu \in \mathcal{M} $. 
\end{definition}

\begin{remark} \label{McompT2} 
By Theorem \cite[Th. 2.4]{Butler:WkConv},  a uniformly bounded in variation family of (deficient) topological measures (with weak topology) is a 
compact Hausdorff space.
\end{remark}

The remaining results and examples in this section are from \cite[Sect. 3, Sect. 5]{Butler:QLM}.

\begin{definition} \label{IT}
Let $X$ and $Y$ be LC spaces.
A map $q: \mathscr{O}(X) \cup \mathscr{K}(X) \longrightarrow \mathscr{O}(Y)  \cup  \mathscr{K}(Y)$ such that
\begin{enumerate}[label=(IT\arabic*),ref=(IT\arabic*)]
\item \label{IT1}
$q(U) \in \mathscr{O}(Y)$  for $ U  \in \mathscr{O}(X)$, and $ q(K) \in  \mathscr{K}(Y)$ for $ K \in\mathscr{K}(X)$;
\item \label{IT3}
$q(U) =  \bigcup \{q(K): K \subseteq U, K \in \mathscr{K}(X) \}$ for $ U \in \mathscr{O}(X)$; 
\item \label{IT4}
$ q(K) = \bigcap \{ q(U): K \subseteq U, U \in \mathscr{O}(X) \}$ for $K \in \mathscr{K}(X)$;
\item \label{IT2} 
$q( A \sqcup B)  = q(A) \sqcup q(B)$;   
\end{enumerate}
is called an image transformation (from X to Y) if  $A, B, A\sqcup B \in \mathscr{O}(X) \cup \mathscr{K}(X)$; 
$q$ is called a d-image transformation if $A, B \in \mathscr{K}(X)$.
If $X = Y$ we call $q$ a (d-) image transformation on $X$.
\end{definition}

\begin{lemma} \label{ITsvva}
Let $X$ and $Y$ be locally compact spaces.
Suppose $q: \mathscr{O}(X) \cup \mathscr{K}(X) \longrightarrow \mathscr{O}(Y)  \cup  \mathscr{K}(Y)$ is a (d-) image transformation. 
The following holds:  
\begin{enumerate}[leftmargin=0.25in, label=(\roman*),ref=(\roman*)]
\item \label{ITsv0}
$q(\emptyset) = 0$.
\item \label{ITsv1}
$q$ is  monotone, i.e. if $ A \subseteq B$, $A, B \in  \mathscr{O}(X) \cup \mathscr{K}(X) $  then $ q(A) \subseteq q(B)$.
\item \label{ITsv2}
If $ K \subseteq q(U), K \in  \mathscr{K}(Y), U \in \mathscr{O}(X)$ then there exists $ C \in \mathscr{K}(X)$ 
such that $ C \subseteq U$ and $ K \subseteq q(C) \subseteq q(U)$.
\item \label{ITsv3}
If a net $ U_{t} \nearrow U $, where $U_t, U$ are open sets, then $q(U_{t}) \nearrow q(U)$.
\item \label{ITsv4}
If a net $K_{\alpha} \searrow K$, where $ K_{\alpha},  K$ are compact sets, then $q(K_{\alpha}) \searrow q(K)$.
\item \label{ITsv5}
If  $q(K) \subseteq W, K \in \mathscr{K}(X), W \in  \mathscr{O}(Y)$ then there exists $U \in \mathscr{O}(X)$ 
such that $ K \subseteq U$ and $ q(K) \subseteq q(U) \subseteq q(\overline U) \subseteq W$.
\end{enumerate}
\end{lemma} 

\noindent
Let  $X^{\flat}$ (resp.,  $X^\sharp$) denote the set of all simple topological (resp., deficient topological) measures on $X$. 
For $ A \subseteq X$, let $A^\flat = \{ \mu \in X^{\flat}: \mu(A) = 1\}$, and $A^\sharp = \{ \mu \in X^{\sharp}: \mu(A) = 1\}$. 
For a LC space $X$, spaces $X^{\flat}$,  $X^\sharp$, $K^\flat$, and $K^\sharp$ (for any compact $K$) 
are compact Hausdorff. 

\begin{example} \label{staIT}
Let $X$ be LC. For $A \in \mathscr{O}(X) \cup \mathscr{K}(X)$ let $\Lambda(A) = A^{\flat}$ (respectively,  $\Lambda(A) = A^\sharp$).
Then  $\Lambda$ is a d-image transformation from $X$ to $X^{\flat}$ (resp., an image transformation from $X$ to $X^\sharp$). 
\end{example}

\begin{definition}  \label{q*}
Let $q$ be an image transformation (respectively, a d-image transformation)  from $X$ to $Y$. The adjoint map $q^* : TM(Y) \rightarrow TM(X)$
(resp.,   $q^* : DTM(Y) \rightarrow DTM(X)$)  is given by $q^*\nu = \nu \circ q$, i.e $q^*\nu(A) = \nu(q(A))$ for $ A \in \mathscr{O}(X) \cup \mathscr{K}(X)$.
\end{definition}

\noindent
It is easy to verify that $(p \circ q)^* = q^* \circ p^*$ for (d-) image transformations $p$ and $q$.

\begin{definition}
A function between topological spaces is proper if inverse images of compact subsets are compact.
\end{definition}

\begin{example}  \label{invfunIT}
Let $u :Y  \longrightarrow X$ be a continuous proper function.  
If $Y$ is compact, then any continuous $u :Y  \longrightarrow X$ is proper.    
Let $q(A) = u^{-1}(A)$ for $A \in \mathscr{O}(X) \cup \mathscr{K}(X)$. 
Then $q$ is a (d-) image transformation.
A (deficient) topological measure $ \nu$ on $Y$ 
induces a (deficient) topological measure $\mu = \nu \circ u^{-1}: A \mapsto \nu(u^{-1}(A))$ on $X$ (see 
\cite[Pr. 5.1]{Butler:DTMLC},  \cite[Ex. 14]{AarnesJohansenRustad}.)
\end{example}

\begin{theorem} \label{adjCont}
Suppose $X,Y$ are LC and $q$ is a d-image transformation (respectively, an image transformation) from $X$ to $Y$.
Then the adjoint map $q^* : (DTM(Y), wk) \rightarrow (DTM(X), wk) $ (resp.,  $q^* : (TM(Y), wk) \rightarrow (TM(X), wk) $) 
preserves finite linear combinations and  
is continuous (i.e. $ \mu_t \Longrightarrow \mu$ implies $q^* \mu_t \Longrightarrow  q^*\mu$).
In particular,  $q^* :Y^{\flat} \rightarrow X^{\flat}$ (resp.,  $q^* :Y^\sharp \rightarrow X^\sharp$)  is continuous.
If $y \in q(\{ x\})$ then $q^*(\delta_y) = \delta_x$.
\end{theorem}

\begin{definition} \label{qhDef}
Suppose $X,Y$  are LC and $E$ is a vector space over $\mathbb{R}$ equipped with a norm or an extended norm. 
A map $\theta: C_0(X) \longrightarrow C_0(Y) $ is called a quasi-homomorphism if it is positive (i.e. $ \theta(f) \ge 0$ for each $ f \ge 0$),
linear and multiplicative  and  on each subalgebra $B(f)$, $f \in C_0(X)$.
$\theta: C_0(X) \longrightarrow C_0(Y) $ is called a conic quasi-homomorphism if 
(i) $0 \le g \le f$ implies $\theta(g) \le \theta(f)$;  (ii) $f\, g=0, \, f, g  \ge 0$ implies $ \theta(f+ g) = \theta(f) + \theta(g)$; 
(iii) $\theta(a g + bh) = a \theta(g) + b \theta(h)$ for $g,h \in A^+(f), \ a,b \ge 0$ (for each $f \in C_0(X)$);  
(iv) $\theta$ is multiplicative on each cone $A^+(f)$.
\end{definition} 

The next theorem is a part of a theorem in \cite[Sect. 5]{Butler:QLM}.

\begin{theorem} \label{ITqh}
Suppose $X$ and $Y$  are LC. There is a 1-1 correspondence between d-image transformations (respectively, image transformations) 
$q$ from $X$ to $Y$ and 
conic quasi-homomorphisms $\theta : C_0^+(X) \longrightarrow C_0^+(Y)$ 
(resp., quasi-homomorphisms $\theta : C_0(X) \longrightarrow C_0(Y)$),
and there is a 1-1 correspondence between d-image transformations (resp., image transformations) 
$q$ and continuous k-proper functions $w: Y \longrightarrow X^{\flat}$ (resp.,   $w: Y \longrightarrow X^\sharp$).
Moreover,
\begin{enumerate}[leftmargin=0.25in, label=(h\arabic*),ref=(h\arabic*)]
\item \label{ITqh1}
$\theta(f)(y) = \int_X f\, dw_y =  \int_X f \, d(q^*\delta_y)$ (where $w_y = w(y)$).
\item \label{ITqh6}
$ \| \theta(f) \| \le \| f\|$. 
\item \label{ITqh6a}
If $ f, g \in C_c^+(X)$ then $ \| \theta(f) - \theta(g) \| \le \| f -g\|$, and if $X$ is compact, then 
$ \| \theta(f) - \theta(g) \| \le \| f -g\|$ for any $f,g \in C(X)$. 
For a quasi-homomorphism, $\| \theta(f) - \theta(g) \| \le 2 \| f -g\|$. 
\item \label{ITqh4}
For an image transformation $q$ and an open set $ A \subseteq \mathbb{R}$ or a closed set $A \subseteq \mathbb{R} \setminus \{ 0\}$ 
\begin{align} \label{qthetaf}
q(f^{-1}(A)) = (\theta(f))^{-1} (A),
\end{align}
and for compact $X$ this holds for any open or closed $ A \subseteq \mathbb{R}$. For a d-image transformation $q$ equality (\ref{qthetaf}) holds for
$A=[t, \infty)$ and $A=(t, \infty)$ for  almost every $t$ .
\item \label{ITqh12}
$q(X) = Y$.
\item \label{ITqh15}
$(q(E))^{\flat}  = (q^*)^{-1}(E^{\flat})$ (resp., $(q(E))^{\sharp}  = (q^*)^{-1}(E^{\sharp})$ for a compact or an open set $E \subseteq X$.
\item \label{ITqh5}
$ \int_X f \, dq^*\nu = \int_Y \theta(f) \, d\nu \  $.
\end{enumerate}
\end{theorem}

\begin{example} \label{consQst}
Let $X$ be locally compact, $Y$ be compact, and  $\mu$ be a simple (deficient) topological measure on $X$. 
For $A \in \mathscr{O}(X) \cup \mathscr{K}(X)$ let $q(A) = \emptyset$ if $ \mu(A) = 0$, and let $q(A) = Y$ if $ \mu(A) = 1$.
Then $q$ is a (d-) image transformation from $X$ to $Y$. If $ \nu$ is a normalized (deficient) topological measure on $Y$ 
then for any open set $V$ in $X$,
$(q^*\nu)(V) = \nu(q(V)) = \nu(Y) = 1$ if $ \mu(V) =1$, and $(q^*\nu)(V) = 0$ if $ \mu(V) =0$. In other words, $q^*\nu = \mu$.
Thus, on  normalized (deficient) topological measures the adjoint map $q^*$ is a constant map.
\end{example} 

\begin{theorem}  \label{ITsolid}
Let $X$ be a LCH connected locally connected space, $Y$ be LCH space. 
Suppose $q: \mathscr{A}_{s}^{*}(X) \longrightarrow  \mathscr{O}(Y) \cup \mathscr{K}(Y)$ such that:
\begin{enumerate}[label=(\roman*),ref=(\roman*)]
\item \label{ITsol1}
$ \bigsqcup_{i=1}^n q(C_i) \subseteq q(C)$ whenever $\bigsqcup_{i=1}^n C_i \subseteq C,  \ \  C, C_i \in \mathscr{K}_{s}(X)$; 
\item  \label{ITsol2}
$q(U) \in \mathscr{O}(Y)$ and  $q(U) = \bigcup \{ q(K): \ K \subseteq U , \ K \in \mathscr{K}_{s}(X) \}$ for $U \in \mathscr{O}_{s}^{*}(X)$; 
\item \label{ITsol3}
$q(K) \in \mathscr{K}(Y)$ and  $q(K) = \bigcap \{ q(U) : \  K \subseteq U, \ U \in \mathscr{O}_{s}^{*}(X) \}$ for $ K  \in \mathscr{K}_{s}(X)$; 
\item  \label{Qsolidparti}
$ q(A) = \bigsqcup_{i=1}^n q (A_i)$ whenever $A = \bigsqcup_{i=1}^n A_i, \ \ A , A_i  \in \mathscr{A}_{s}^{*}(X)$.
\end{enumerate}
Then $q$ extends uniquely to an image transformation from $X$ to $Y$. 
\end{theorem}

\begin{remark} \label{easysolpar}
If $X$ is compact,  in Theorem \ref{ITsolid}  it is enough to consider one of the equivalent conditions \ref{ITsol2} and \ref{ITsol3}.  
If in Theorem \ref{ITsolid} $X$ is a noncompact whose one-point compactification has genus $0$,
then part \ref{Qsolidparti} holds automatically by \cite[L. 15.2]{Butler:TMLCconstr}.  
If $X$ is a q-space, then for Theorem \ref{ITsolid}\ref{Qsolidparti} we only need to show that 
$q(X) = q(A) \sqcup q(X \setminus A)$, where $A$ is a closed solid or an open solid set.
\end{remark}

\begin{example} \label{3pts2vIT}
Let $X = Y$ be the unit square in $\mathbb{R}^2$. Let $E = \{x, z\}$, where $x, z\in X$. As in \cite[p. 46]{AarnesGrubb}, we consider the following map $q$.
For an open or closed solid set $A$ define $q(A) = 0$ if $| A \cap E| = 0$; $q(A) =A $ if $| A \cap E| = 1$; $q(A) = X$ if $| A \cap E| = 2$. 
Then $q$ satisfies Theorem \ref{ITsolid}, hence, extends to an image transformation from $X$ to $X$.
For the corresponding function $w$, each $w(y)$ is the topological measure from Example \ref{nvssf} for 
the set $P = \{ x,z,y\}$.
\end{example}

\noindent
For more examples of image transformations see \cite[Sect. 3]{Butler:QLM}.
 
\section{Image Transformation Systems and Markov operators} \label{SectMarkov}

A well-known result of Hutchinson gives the existence of the invariant measure for the Markov  operator  
$M(\nu) = \sum_{i=1}^n \alpha_i \, \nu \circ u_i^{-1}$ associated with the IFS (iterated function system) with probabilities $(X, u_i, \alpha_i, i=1, \ldots, n)$
(see   \cite{Hutchinson}, \cite[Ch. IX]{Barnsley}).
Here $u_i$ are contractions on a complete metric space $X,   \alpha_i > 0,  \sum_{i=1}^n \alpha_i = 1$, and the Markov operator is defined on the set 
 $\mathscr{M}^1(X)$ of normalized regular Borel measures on $X$ with bounded support.
The support of the invariant measure is the attractor of the IFS.   

Replacing $u_i^{-1}$ by any image transformations $q_i$ (compare with Example \ref{invfunIT}) 
with a contractivity factor less than 1 on a compact metric space $X$ produces 
an image transformation system (ITS) instead of an IFS. 
The analog of  the Markov operator for this ITS  has the invariant topological measure.  
Moreover, there is a relation between the resulting invariant topological measure and the invariant measure for the IFS  $(X^*, q_1^*, \ldots, q_n^*)$.
See \cite[Pr. 5.2]{Aarnes:ITfirst}, \cite[Pr. 5.8]{Pedersen}, \cite[Pr. 36]{AarnesJohansenRustad}.  

Extending these ideas, we consider (d-) image transformation systems with infinite positive linear combinations of (d-) image transformations 
between locally compact spaces.  We consider associated Markov operators, which are continuous. 
Restricted to the space of measures, such operators produce
Markov-Feller operators with nonlinear duals which are infinite  positive linear combinations of (conic) quasi-homomorphisms.
We show that any uniformly bounded in variation family $\mathcal{M}$ of (deficient) topological measures is a complete metric space with respect to the 
Kantorovich-Rubinstein metric. Any Markov operator on  $\mathcal{M}$ has the unique invariant (deficient) topological measure on this metric space   
if (d-) image transformations are contractions in the Kantorovich-Rubinstein metric.
We show various relationships between a Markov operator and IFS  $(X^*, q_1^*, \ldots, q_n^*)$.   

\begin{definition} \label{contrFact}
Let $(X, d)$ be a metric space, and let $s \ge 0$. We denote by  
$Lip_s(X,d) = \{ f \in C(X): | f(x) - f(y)| \le s d(x,y) \mbox{  for any  } x, y \in X \}$, and by $Lip$ the collection comprised of $f$  belonging to 
$Lip_s$ for some $s$.

A map $u: X \longrightarrow X$ is a contraction if there is $s \in [0,1)$ such that 
$d(u(x), u(y)) \le s d(x,y)$ for any $x,y \in X$. Any such number $s$ is called a contractivity factor for $u$.
\end{definition}

\begin{definition} \label{KRfunction}
For finite topological measures $ \mu, \nu$ on $X$, define the Kantorovich-Rubinstein function $d_{\text{KR}}$ 
\begin{align} \label{KRmetric}
d_{\text{KR}}(\mu, \nu) &= \sup \{| \int_X f \, d\mu -  \int_X f \, d\nu|: f \in Lip_1(X,d) \cap C_c(X)  \}. 
\end{align}
For finite deficient topological measures $ \mu, \nu$ define $d_{\text{KR}}$ 
in the same way, taking nonnegative functions in (\ref{KRmetric}). 
\end{definition} 

\begin{remark}
Definition \ref{KRfunction} is related to the definition of the Kantorovich-Rubinstein metric
for Borel measures, which is obtained from the Kantorovich-Rubinstein norm
$$ \| \mu \|_{KR} = \sup \{ \int_X f \, d\mu: f \in Lip_1(X,d), \, \| f \| \le 1 \}.$$ 
This metric is also sometimes called the Hutchinson metric or 
the Wasserstein metric $W(\mu, \nu)$, although there is no author with the last name Wasserstein. 
L. Vasershtein  (in Russian:
    \textrm{\CYRL.  \CYRV\cyra\cyrs\cyre\cyrr\cyrsh\cyrt\cyre\cyrishrt\cyrn}
) 
and J. Hutchinson used this metric in  \cite{Vasershtein} and \cite{Hutchinson},
but it appeared decades earlier in Kantorovich's  seminal paper \cite{Kantorovich} 
and then in his works with Rubinstein. 
See \cite[pp. 453--454]{Bogachev} and \cite{BogachevKolesnikov} for good information on the history and use of this metric.

Our use of $f \in Lip_1(X,d) \cap C_c(X)$ in (\ref{KRmetric}) is dictated, on one hand, by the relation of (\ref{KRmetric})  
to the Kantorovich-Rubinstein metric for Borel measures and, 
on the other hand, by the role of $C_c(X)$ in the theory of (p-conic) quasi-linear functionals.
\end{remark}

\begin{remark} \label{LipDense}
In a locally compact metric space,  Lipschitz functions with compact support are dense in $C_0(X)$; see, for example, \cite[Th. 2]{Andreou}.
\end{remark}   

\begin{remark}
The Hutchinson proof of the existence of invariant measure is an application of the Fixed Point Theorem to the set  $\mathscr{M}^1(X)$ above 
equipped with the Kantorovich-Rubinstein metric $d_{\text{KR}}$.  
However, the completeness of the metric space $(\mathscr{M}^1(X), d_{\text{KR}})$ is not discussed in \cite{Hutchinson}. 
In fact, if $X$ is unbounded, then $(\mathscr{M}^1(X), d_{\text{KR}})$  is not complete (\cite[As. 1]{Kravchenko}. There are situations
(see, for instance, \cite{Akerlund-Bistrom}, \cite{Kravchenko}) 
where the spaces of measures with some conditions are complete with respect to the Kantorovich-Rubinstein metric. 
This allows for generalization of Hutchinson's 
result regarding invariant measure. 
\end{remark} 

\begin{remark} \label{RmKRd}
For any $  f  \in Lip_1(X,d) \cap C_c(X)$ and any $x, y \in X$, 
$ \int_X f \, d \delta_x - \int_X f \, d \delta_y = f(x) - f(y) \le d(x, y),$
so $d_{KR}(\delta_x, \delta_y) \le d(x,y)$.  
If $X$ is compact, for a $Lip_1(X, d)$ function $g(x) = d(x,y)$ we have  
$d(x,y) = g(x) - g(y) = \int_X g \, d\delta_x - \int_X g \, d\delta_y \le d_{KR}(\delta_x, \delta_y)$, so $d_{KR}(\delta_x, \delta_y) = d(x,y)$.
\end{remark}  

\begin{definition} 
For a set $A$ in a metric space $(X,d)$ define $U_t(A) = \{ x:  d(x, A) < t \}$ for $t >0$ and  $B_t(A) = \{ x:  d(x, A) \le t \}$ for $ t \ge 0$.
For sets $A$ and $B$, define $ \delta(A,B) = \inf \{d(x,y): x \in A, \ y \in B\}$.
\end{definition}

\begin{remark}
$U_t(A)$ is open. 
If $C$ is compact, the set $B_t(C)$ is closed; it is also bounded. 
Thus,  $B_t(C)$  is compact if $X$ is a metric space with the Heine--Borel property; for example,  if $X$ is compact or 
a finite-dimensional linear normed space (with metric induced by the norm).
Note that $\delta $ is not  a metric, but given a compact $K$ and a closed set $G$,  $\delta(K, G) = 0$ iff $ K \cap G \neq \emptyset$. 
\end{remark}

If $F, G $ are closed subsets of $X^\flat$ (respectively, $X^\sharp$), define 
$\delta_{KR}(F,G) = \inf \{ d_{KR}(\mu, \lambda): \mu \in F^\flat, \lambda \in G^\flat (\text{ resp., } \mu \in F^\sharp,  \lambda \in G^\sharp) \} $.   

\begin{proposition} \label{KRstar}
Suppose $(X,d)$ is a metric space. 
If $C$ and $K$ are compacts then  $ \delta_{KR}(C^{\flat}, K^{\flat}) \le \delta_{KR}(C^\sharp,K^\sharp) \le  \delta(C,K)$. 
If  compacts  $C$ and $K$ are (a) intersecting or (b) disjoint and for $t = \delta(C,K)$ at least one of $B_t(C), B_t(K)$ is compact, then 
$ \delta_{KR}(C^{\flat},K^{\flat}) = \delta_{KR}(C^\sharp,K^\sharp) = \delta(C, K)$. 
\end{proposition}

\begin{proof}
Take $x \in C, y \in K$ such that $ \delta(C,K) = d(x,y)$. 
Then $ \delta_{KR}(C^{\flat},K^{\flat}) \le \delta_{KR}(C^\sharp,K^\sharp) \le  \delta_{KR}(\delta_x, \delta_y) \le d(x,y) = \delta(C,K)$.

If $C$ and $K$ intersect then $  \delta_{KR}(C^{\flat},K^{\flat}) =  \delta_{KR}(C^\sharp,K^\sharp) = \delta(C, K) =0$.
Assume $ C \cap K  = \emptyset$ and $B_t(C)$ is compact, where $ t =  \delta(C, K) >0$.
For the last statement we shall show that $ \delta_{KR}(C^{\flat},K^{\flat}) \ge \delta(C, K)$.
For $p < t$,  $B_p(C) $ is compact, $D_p = X \setminus U_p(C)$ is closed, $K \subseteq D_p$, and $C \cap D_p = \emptyset$. 
Let $f(x) = d(x, D_p)$. Then $f \in Lip_1(X, d)$, and $f= 0$ on $D_p$, 
so $ supp \, f \subseteq \overline{U_p}(C) \subseteq B_t(C)$. Thus, $ f \in Lip_1(X, d) \cap C_c(X)$.  

Let $ \mu \in C^{\flat}, \nu \in K^{\flat}$. Since $ f \ge p$ on $C$,  by Remark \ref{RemBRT}\ref{mrDTM}, $ \int_X f\, d\mu \ge p \mu(C) =  p$.
By superadditivity, $ \nu(K) = \nu(X) =1$ implies $\nu(X \setminus K) = 0$, and since $ Coz(f) \subseteq X \setminus K$, we have $ \nu(Coz(f)) = 0$. 
By \cite[Th. 49]{Butler:Integration} $\int_X f \, d\nu = 0$. Then
$ d_{KR} (\mu, \nu) \ge \int_X f \, d\mu -   \int_X f \, d\nu \ge p.$
This holds for any $ p < t$, so $d_{KR} (\mu, \nu) \ge t $. The last inequality holds for any   $ \mu \in C^{\flat}, \nu \in K^{\flat}$, 
so  $\delta_{KR}(C^{\flat},K^{\flat}) \ge t = \delta(C,K)$. 
\end{proof} 

\begin{theorem} \label{sThe}
Let $q$ be a (d-) image transformation from a locally compact metric space $(X, d)$ to a
locally compact metric space $(Y, d')$, 
$\theta$ be the corresponding (conic)
quasi-homomorphism,  and $w$ be as in 
Theorems \ref{ITqh}. Let $ s > 0$. Consider the following conditions:
\begin{enumerate}[label=(s\arabic*),ref=(s\arabic*)]
\item \label{E}
$\theta(Lip_1(X, d) \cap C_c(X)) \subseteq (Lip_s(X, d) \cap C_c(X)$ for a quasi-homomorphism $\theta$; \\
$\theta(Lip_1(X, d) \cap C_c^+(X)) \subseteq (Lip_s(X, d) \cap C_c^+(X)$ for a conic quasi-homomorphism $\theta$.
\item \label{C}
$d_{KR} (q^*\mu, q^*\nu) \le s\, d_{KR} (\mu, \nu)$.
\item \label{D}
$d_{KR} (w(z), w(y)) \le s\, d'(z,y)$ for all $z,y  \in Y$.
\item \label{A}
$\delta (K, G) \le s \delta' (q(K), q(G))$ for any $K, G \in \mathscr{K}(X)$ such that $q(K), q(G) \neq \emptyset$ (where $\delta'$ is defined with respect to the metric $d'$).
\item \label{B}
$\theta(Lip_1(X, d) \cap C_0(X)) \subseteq (Lip_s(X, d) \cap C_0(X)$ for a quasi-homomorphism $\theta$.
\end{enumerate}
Then $ \ref{E} \Longrightarrow  \ref{C} \Longrightarrow \ref{D}$. 
For a quasi-homomorphism  $ \ref{A} \Longrightarrow  \ref{B}$ and  $ \ref{B} \Longrightarrow  \ref{E}$. 
If $(X,d)$ is a metric space such that for any disjoint compacts $C$ and $K$ at least one of $B_t(C), B_t(K)$ is compact  for $ t = \delta(C,K)$, then 
$\ref{D} \Longrightarrow  \ref{A}$; so in this case  for a quasi-homomorphism $\ref{E} -\ref{B}$ are equivalent. 

\end{theorem}

\begin{proof}
$\ref{E} \Longrightarrow \ref{C}$. 
Note that, with a slight abuse of notation, 
on the right hand side of \ref{C} $d_{KR}$ is considered for (deficient) topological measures on $Y$, while on the left hand side of  \ref{C} and in 
\ref{D} $d_{KR}$ is considered for  (deficient) topological measures on $X$.
Let $\mu, \lambda$ be topological  (respectively, deficient topological) measures  on $Y$, 
$ f \in Lip_1(X, d) \cap C_c(X)$, (respectively,  $ f \in Lip_1(X, d) \cap C_c^+(X)$), and $s >0$.
By \ref{E}  $\frac1s \theta(f)$ is in  $ Lip_1(X, d) \cap C_c(Y)$ (respectively, in $Lip_1(X, d) \cap C_c^+(Y) )$,   
so using Theorem \ref{ITqh}\ref{ITqh5}  we have:
\begin{align*}
\int_X f\, dq^*\mu -  \int_X f\, dq^*\lambda &= \int_Y \theta(f) \, d\mu -  \int_Y \theta(f) \, d\lambda \\
&= s \left( \int_Y \frac{\theta(f)}{s} \, d\mu -  \int_Y \frac{\theta(f)}{s} \, d\lambda \right) \le s\, d_{KR}(\mu, \lambda).
\end{align*} 
Thus, $d_{KR} (q^*\mu, q^*\nu) \le s\, d_{KR} (\mu, \nu)$.

$\ref{C}  \Longrightarrow  \ref{D}$.
By Theorem \ref{ITqh}\ref{ITqh1}  and  Remark \ref{RmKRd},
$d_{KR}(w(z), w(y)) = d_{KR}(q^*\delta_z, q^*\delta_y)  \le s\, d_{KR}(\delta_z, \delta_y) \le s\, d'(z,y).$ 

$\ref{D}  \Longrightarrow  \ref{A}$. 
Suppose $(X,d)$ is a metric space such that for disjoint compacts $C$ and $K$ at least one of $B_t(C), B_t(K)$ is compact for $ t = \delta(C,K)$.
Suppose $K, G \in \mathscr{K}(X)$ and $q(K), q(G) \neq \emptyset$. Let $y \in q(K), z \in q(G)$. 
Since $w(y)(K) = \delta_y(q(K)) = 1$, we have $w(y) \in  K^\sharp$ (respectively, $w(y) \in  K^\flat$).  
Also, $w(z) \in G^\sharp$ (resp.,  $w(z) \in G^\flat$).
By part \ref{D}  we have $\delta_{KR} (K^\sharp, G^\sharp) (\text{ resp., } \delta_{KR} (K^\sharp, G^\sharp) \le  d_{KR} (w(y), w(z)) \le s d'(y,z)$. 
Then by Proposition \ref{KRstar} 
$\delta(K, G) \le s d'(y,z)$. 
This holds for any  $y \in q(K), z \in q(G)$, so $\delta(K, G) \le s \delta' (q(K), q(G))$.

$\ref{A} \Longrightarrow \ref{B}$. Suppose $ \theta$ is a quasi-homomorphism. 
Let $ f \in Lip_1(X,d) \cap C_0(X)$. 
We need to show that $ |\theta(f) (y_1) - \theta(f) (y_2)| \le s\, d'(y_1, y_2)$.
If $\theta(f) (y_1) = \theta(f) (y_2) = 0$, there is nothing to prove. Assume that $ \alpha_i =  \theta(f) (y_i) \neq 0$ for $i=1,2$.
By Theorem \ref{ITqh}\ref{ITqh4} for $i=1,2$ the set $ q( f^{-1}(\{ \alpha_i \} )) = (\theta(f))^{-1}(\{ \alpha_i\}) $ is a
nonempty compact, as it contains $y_i$. Then $f^{-1}(\{ \alpha_i \} ) \neq \emptyset$, and we
choose $x_i \in  f^{-1}(\{ \alpha_i\}),  i=1,2$ such that $\delta( f^{-1}(\{\alpha_1\}), f^{-1}(\{ \alpha_2\}) = d(x_1, x_2)$.
Then
\begin{align*}
|\theta(f) (y_1) &- \theta(f) (y_2)|  = |\alpha_1 - \alpha_2| = | f(x_1) - f(x_2)|  \le d(x_1, x_2)  \\
& = \delta( f^{-1}(\{\alpha_1\}), f^{-1}(\{ \alpha_2\})  \le s \delta'( q(f^{-1}(\{\alpha_1\})), q(f^{-1}(\{ \alpha_2\})) \le s d'(y_1, y_2).
\end{align*}

Now assume that  $ \alpha = \theta(f) (y_1) \neq 0$, and $ \theta(f) (y_2) = 0$. Let $ \epsilon < | \alpha|$, $V = ( -\epsilon, \epsilon)$.
By Theorem \ref{ITqh}\ref{ITqh4}  $y_2 \in q(f^{-1}(V)) = (\theta(f))^{-1}(V)$.
By Lemma \ref{ITsvva} there is a nonempty compact $C \subseteq f^{-1}(V)$ such that $y_2 \in q(C) \subseteq q(f^{-1}(V))$.
Since $y_1 \in   (\theta(f))^{-1}(\{\alpha\}) = q( f^{-1}(\{\alpha\} )$, the last set is  a nonempty compact, and so 
$ \delta( f^{-1}(\{\alpha\}), C ) \le s \delta'(q(f^{-1}(\{\alpha\}), q(C)) \le s d'(y_1, y_2).$
Compact $ f^{-1}(\{ \alpha\})$ is nonempty (otherwise, $q( f^{-1}(\{\alpha\} )$ would be empty as well), so
we choose $x_1 \in f^{-1}(\{ \alpha\}), x_0 \in C$ for which $\delta( f^{-1}(\{\alpha\}), C )  = d(x_1, x_0) $. Then 
$ \delta( f^{-1}(\{\alpha\}), C )  = d(x_1, x_0)  \ge | f(x_1) - f(x_0) | \ge |\alpha| - \epsilon.$
It follows that
$ s d'(y_1, y_2) \ge  | \alpha| =   |\theta(f) (y_1)| =  |\theta(f) (y_1) - \theta(f) (y_2)|. $

$\ref{B}  \Longrightarrow  \ref{E}$ Clear, since for a quasi-homomorphism, $\theta(C_c(X)) \subseteq C_c(Y)$ by \cite[Th. 65, Sect. 4]{Butler:QLM}.

\end{proof} 

\begin{definition}
If a (conic) quasi-homomorphism $\theta$ satisfies \ref{E} of Theorem \ref{sThe} with $ s <1$, 
we say $\theta$ and  a corresponding (d-) image transformation $q$ are
contractions, and $s$ is a contractivity factor of $\theta$ and $q$. 
\end{definition} 

\begin{theorem} \label{KRisMetr}
Let $\mathcal{M} $ be a uniformly bounded in variation family of (deficient) topological measures on a locally compact bounded metric space  $X$.
Then  $d_{\text{KR}}$ is a metric on $\mathcal{M} $.
If a net $(\mu_{\alpha}) \subseteq \mathcal{M} $, $\mu  \in  \mathcal{M} $, and $ d_{\text{KR}} ( \mu_{\alpha}, \mu) \longrightarrow 0$, 
then $ \mu_{\alpha} \Longrightarrow \mu$. 
If $X$ is compact, then the topology on $ \mathcal{M}$ induced  by the metric $ d_{\text{KR}}$ is the weak topology.
\end{theorem}

\begin{proof}
Suppose $X$ is contained in a ball $B(x_0, r)$ and $\mu(X) \le M$ for every $ \mu \in \mathcal{M} $.
Take (a nonnegative) $ f \in Lip_1(X,d) \cap C_c(X) $, and let 
$y, z$ be such that $f(y) =0, | f(z) | = \| f\|$. Then $\| f \|  =  |f(z) - f(y)|  \le d(z, y)  \le 2r$, and  
so by (\ref{estin}) $d_{\text{KR}} (\mu, \nu) \le 4rM$. Thus, $d_{\text{KR}}$ is a real-valued map. 

We shall check that  $d_{\text{KR}} (\mu, \nu) = 0 $ implies $\mu = \nu$ (other axioms of a metric are obvious).
If $ \mu \neq \nu$ then by (\ref{EqByInt}) there is $ f \in C_c^+(X)$ such that $ \int_X f \, d \mu  \neq \int_X f \, d \nu$, 
and so  $d_{\text{KR}} (\mu, \nu) > 0 $. 
Thus,  $d_{\text{KR}}$ is a metric.

Now suppose $ d_{\text{KR}} ( \mu_{\alpha}, \mu) \longrightarrow 0$. Let $ U \in \mathscr{O}(X), \epsilon >0$. 
For  $ f \in C_c^+(X)$ satisfying  $ supp \, f  \subseteq U, \, f \le 1_U, \,  \mu(U) <  \int_X f \, d\mu  + \epsilon$, 
by (\ref{LipCc}) pick a compactly supported Lipschitz function $h$  such that 
$ | \int_X f \, d \lambda - \int_X h \, d \lambda | < \epsilon$ for any $ \lambda \in  \mathcal{M} $. 
From  $d_{\text{KR}} ( \mu_{\alpha}, \mu) \longrightarrow 0$
we have $ | \int_X h \, d \mu_{\alpha} - \int_X h\, d \mu | < \epsilon $ for all $ \alpha \ge \alpha_0$. 
Then  for  $ \alpha \ge \alpha_0$
$$ \mu_{\alpha} (U) \ge   \int_X f\, d \mu_{\alpha} \ge   \int_X h \, d \mu_{\alpha} - \epsilon >  
\int_X h \, d \mu - 2 \epsilon  \ge  \int_X f\, d \mu - 3 \epsilon \ge \mu(U) - 4 \epsilon.$$
It follows that $ \liminf \mu_{\alpha} (U) \ge \mu(U)$. Similarly one can show that $ \limsup \mu_{\alpha} (K) \le \mu(K)$ for any compact $K$.  
By \cite[Th. 2.1]{Butler:WkConv},  $ \mu_{\alpha} \Longrightarrow \mu$.  
The last statement now can be proved as in \cite[Prop. 1.10]{DickZap}. 
\end{proof}

\begin{theorem} \label{completeMS} 
Suppose $X$ is a locally compact bounded metric space, and $\mathcal{M}$ is a uniformly bounded in variation 
family of (deficient) topological measures on $X$. Then $(\mathcal{M}, d_{\text{KR}})$ is a complete metric space.
\end{theorem}

\begin{proof}
By Theorem \ref{KRisMetr} $d_{\text{KR}}$ is a metric on $\mathcal{M}$.   
Suppose $(\mu_n)$ is a fundamental sequence in $(\mathcal{M}, d_{\text{KR}}). $ 
For a compactly supported Lipschitz function $g$  the sequence $( \int_X g\, d \mu_n)$ is fundamental, 
and we define $L(g) = \lim_{n \rightarrow \infty} \int_X g\, d \mu_n$.
For $f \in C_0(X)$ and a sequence $(g_i)$ of compactly supported Lipschitz functions with 
$\| f - g_i \| \rightarrow 0$ (such a sequence exists by Remark \ref{LipDense}), 
the sequence $ (L(g_i))$ is fundamental (use (\ref{LipCc})), so we define $L(f) = \lim_{i \rightarrow \infty} L(g_i)$.
If $(g_i)$ and $(h_i)$ are two sequences converging to $f$, where $ g_i, h_i \in Lip  \, \cap C_c(X)$, then 
by (\ref{LipCc})   $\lim_{i \rightarrow \infty} L(g_i) = \lim_{i \rightarrow \infty} L(h_i)$, and $L$ is well-defined on $C_0(X)$.

We shall show that $L$ is a (p-conic) quasi-linear functional. Suppose $f \, g = 0, \, f,g \in C_0^+(X)$. We may choose 
compactly supported Lipschitz functions $f_i,  g_i  \ge 0$ converging to $f, g$ respectively, such that  $f_i \, g_i = 0$.
Then by Definition \ref{cqlf} $L(f_i + g_i) =L(f_i) + L(g_i)$, and so $L(f+g) =  \lim_{i \rightarrow \infty} L(f_i + g_i)  = L(f) + L(g)$. 
If $0 \le g \le f, \ f,g \in C_0(X)$ then we may choose $g_i, f_i$ such that $ 0 \le g_i \le f_i$, so $L(g_i) \le L(f_i)$, which gives $L(g) \le L(f)$.
Now suppose $ \phi \circ f, \psi \circ f \in A^+(f)$ where $ \phi, \psi$ continuous (nondecreasing) functions on $[a,b] = \overline{f(X)}$, 
and  $ \phi(0) = \psi(0) = 0$ if $X$ is noncompact. Pick a sequence $(\phi_j)$ of (nondecreasing) Lipschitz functions on $[a,b]$, 
$\phi_j(0) =0$ if $X$ is noncompact so that $\| \phi_j - \phi \| \rightarrow 0$ (for instance, one may use piecewise linear functions as $\phi_j$).    
Using uniform continuity of $ \phi$ and going to subsequences if necessary, we have $ \| \phi_k \circ f_k - \phi \circ f \| \rightarrow 0$, 
and each $\phi_k \circ f_k \in Lip \, \cap C_c(X)$. 
Similarly approximate $ \psi \circ f$ by $ \psi_k \circ f_k$. 
Since $L(  \phi_k \circ f_k  +  \psi_k \circ f_k) = L( \phi_k \circ f_k) + L( \psi_k \circ f_k)$, 
we have $L(\phi \circ f + \psi \circ f) = L(\phi \circ f) + L( \psi \circ f)$.
Thus, $L$ is a (p-conic) quasi-linear functional. 

Let $ \mu$ be the (deficient) topological measure corresponding to $L$.
We shall show that  $d_{\text{KR}}(\mu_n, \mu) \rightarrow 0$ as $ n \rightarrow \infty$.
For $ \epsilon >0$ let $n_\epsilon$ be such that $d_{\text{KR}}(\mu_n, \mu_k) < \epsilon$ for any $n,k \ge n_\epsilon$.
Let $f \in Lip \cap C_c(X)$.
Choose $k \ge n_\epsilon$ such that $  | \int_X f \, d\mu_k- L(f)| =  | \int_X f \, d\mu_k- \int_X f\, d \mu| < \epsilon.$ Then for any $n \ge n_\epsilon$ 
$$  |  \int_X f \, d\mu_n - \int_X f\, d \mu |  \le  | \int_X f \, d\mu_n - \int_X f\, d \mu_k | +  | \int_X f \, d\mu_k - \int_X f\, d \mu |   
\le d_{\text{KR}}(\mu_n, \mu_k) + \epsilon < 2 \epsilon.
$$
Thus, $d_{\text{KR}}(\mu_n, \mu) \le 2 \epsilon$.
\end{proof}

From Remark \ref{RemBRT} we see that if $\mu = \alpha_1 \mu_1 + \alpha_2 \mu_2$, where $\alpha_1, \alpha_2 \ge 0$ 
then $ \int f \, d\mu = \alpha_1  \int f \, d\mu_1 +  \alpha_2  \int f \, d\mu_2$ for any $ f \in C_0(X)$ 
and any finite deficient topological measures $\mu_1, \mu_2$. But we can say more. 
  
\begin{lemma} \label{inftsumD}
Suppose $X$ is locally compact,  $\sum_{i=1}^\infty \mu_i(X) < \infty$ where each $\mu_i$ is a (deficient) topological measure. 
Then $\mu = \sum_{i=1}^\infty \mu_i$ is a finite (deficient) topological measure. 
Moreover, $ \int f \, d \mu = \sum_{i=1}^\infty  \int f\, d \mu_i$ for any $ f \in C_0(X)$.
\end{lemma}

\begin{proof}
The first statement is \cite[L. 5.1]{Butler:WkConv}. Let $ f \in C_0(X), f(X) \subseteq [a,b]$. 
For $ \epsilon >0$ let $N$ be such that
$\sum_{i=N+1}^\infty \mu_i(X) < \epsilon$, and let $\mu' = \sum_{i=1}^N \mu_i$.  By Remark \ref{RemBRT}\ref{prt1} we see that 
$R_{1, \mu, f}(t) \le R_{1, \mu', f}(t) + \epsilon$, and  
$$  \int_X f \, d\mu \le  \int_X f \, d\mu'  + b \epsilon. $$
Then using formula (\ref{estin}) we have:
$$ | \int_X f \, d\mu -  \sum_{i=1}^\infty  \int f\, d \mu_i |  \le |  \int_X f \, d\mu -  \int_X f \, d\mu' |  + | \sum_{i=N+1}^\infty  \int f\, d \mu_i | 
 \le b \epsilon + \| f \| \epsilon, $$
which finishes the proof.
\end{proof}

\begin{corollary} \label{inftSuDTM}
Suppose $X$ is locally compact, 
family $\mathcal{M}$ is a uniformly bounded in variation family of (deficient) topological measures on $X$, and 
$\sum_{i=1}^\infty \alpha_i \le 1, \, \alpha_i \ge 0$.
Then $\mu = \sum_{i=1}^\infty  \alpha_1 \mu_i$  belongs to $\mathcal{M}$,  and 
$ \int f \, d \mu = \sum_{i=1}^\infty  \alpha_i  \int f\, d \mu_i$ for any $ f \in C_0(X)$.  
\end{corollary}

\begin{definition} \label{DefMarkovOp}
Let $X,Y$ be locally compact. 
An operator $S: DTM(Y) \longrightarrow DTM(X)$ is called a Markov operator if 
$ S(a \mu + b \nu) = a S(\mu) + b S(\nu)$ for $a,b \ge 0, \mu, \nu \in DTM(Y)$.
A Markov operator $S: DTM(Y) \longrightarrow DTM(X)$ is called a Markov-Feller operator if there is an operator
$T: C_0(X) \longrightarrow C_0(Y)$ such that $ \int_X f \, d(S(\mu)) = \int_Y T(f) \, d\mu $ for $g \in C_0(X), \mu \in DTM(Y)$.  
The operator $T$ is called dual to $S$.
\end{definition}

\noindent
Definition \ref{DefMarkovOp} is closely related to \cite[p. 345]{LasotaMyjakSzarek}, where these operators are defined for spaces of measures. 

\begin{definition}
Let $X,Y$ be locally compact. 
A (d-) image transformation system is  
$\{ X,  Y, q_i^*, \alpha_i \} $ where $q_i$ is a (d-) image transformation from $X$ to $Y$, 
$\alpha_i \ge 0$ for $ i = 1, 2, \ldots$ and $\sum_{i=1}^\infty \alpha_i \le 1$. 
If $X =Y$ we write $\{ X, q_i^*, \alpha_i \} $. 
An image transformation system has the contractivity factor  $s$ if each $q_i$ has the  contractivity factor $ s_i \le s $.  
We associate with the (d-) image transformation system $\{ X, Y, q_i^*, \alpha_i \} $ an operator 
$S:TM(Y) \longrightarrow TM(X)$ (resp.,  $S: DTM(Y) \longrightarrow DTM(X))$  
defined by
\begin{align} \label{MarkovS} 
S(\mu) = \sum_{i=1}^\infty \alpha_i q_i^* (\mu). 
\end{align}
\end{definition} 

\begin{remark}
If $X=Y$ is a compact metric space, $ \mu$ is a measure, $\sum_{i=1}^n \alpha_i = 1, \,  \alpha_i > 0$,
and $q_i= u^{-1}$ is the inverse of a contraction $u_i$ on $X$ for $i=1, \ldots, n$, then we obtain the well-known Markov operator associated with the
IFS with probabilities $(X, u_1, \ldots, u_n, \alpha_1, \ldots, \alpha_n)$  (see, for example, \cite[Ch.9, \S 6]{Barnsley}). 
\end{remark}

\begin{theorem} \label{ITcontin}
Suppose $X, Y$ are locally compact metric spaces,   $\mathcal{M}$ is a uniformly bounded in variation family of (deficient) topological measures on $Y$. 
Let $\{ X, Y, q_i^*, \alpha_i \} $ be a (d-) image transformation system with  $\alpha_i \ge 0$ and $\sum_{i=1}^\infty \alpha_i < \infty$.
Then the operator $S: (\mathcal{M}, wk) \longrightarrow (TM(X), wk)$ (resp.,  $S: (\mathcal{M}, wk) \longrightarrow (DTM(X), wk)$) 
given by (\ref{MarkovS}) 
is a continuous Markov operator which preserves countable convex linear combinations. 
The image of $S$ is a uniformly bounded in variation family of (deficient) topological measures on $X$. 
On the collection of measures from $\mathcal{M}$, $S$ is a Markov-Feller operator with nonlinear
dual operator $T = \sum_{i=1}^\infty \alpha_i \theta_i$, where each $\theta_i$ is a (conic) quasi-homomorphism corresponding to $q_i$. 
\end{theorem}

\begin{proof}
Let  $ \| \mu \|  \le L$ for each $ \mu \in \mathcal{M}$. Then $S(\mu)(X) \le L  \sum_{i=1}^\infty \alpha_i $, 
so $S(\mathcal{M})$ is a uniformly bounded in variation family of (deficient) topological measures on $X$.
By Corollary \ref{inftSuDTM}, for $ f \in C_0(X)$ 
\begin{align} \label{INTdSmu}
 \int_X f \, d(S(\mu)) = \sum_{i=1}^\infty \alpha_i \int_X f \, d(q_i^* \mu).
\end{align}

By Corollary \ref{inftSuDTM},  a countable convex linear combination of members of $\mathcal{M}$ is in $\mathcal{M}$, 
and it is easy to see that $S$ preserves countable convex linear combinations. 
Suppose $ \mu_t \Longrightarrow \mu$. 
Let $ f \in C_0(X)$. For $ \epsilon >0$ pick $n$ such that 
$ L \| f \| \sum_{i=n+1}^{\infty} \alpha_i  < \epsilon$. 
Then $| \sum_{i=n+1}^\infty \alpha_i \int_X f \, d(q_i^* \nu)| < \epsilon$ for any $ \nu \in  \mathcal{M} $.
By Theorem \ref{adjCont} each  $q_i^*$ is continuous, so there is $ t_0$ such that 
$ \sum_{i=1}^n \alpha_i  |  \int_X f \, d(q_i^* \mu_t) -  \int_X f \, d(q_i^* \mu) | < \epsilon$ for $ t \ge t_0$.
Then for  $ t \ge t_0$
\begin{align*}
| \int_X f \, d(S(\mu_t)) -  \int_X f \, d(S(\mu)) | &= 
|\sum_{i=1}^{\infty} \alpha_i \int_X f \, d(q_i^*\mu_t) -  \sum_{i=1}^{\infty} \alpha_i \int_X f \, d(q_i^*\mu) | \\
& \le   \sum_{i=1}^n |  \alpha_i   \int_X f \, d(q_i^* \mu_t) -  \int_X f \, d(q_i^* \mu) | + 2 \epsilon < 3 \epsilon. 
\end{align*}
Thus,  $S( \mu_t) \Longrightarrow  S(\mu)$, and $S$ is continuous. 

Now let $ \mu$ be a measure from $\mathcal{M}$.
From formula (\ref{INTdSmu}) and Theorem \ref{ITqh}\ref{ITqh5} we have:
$$ \int_X f \, d(S(\mu)) = \sum_{i=1}^\infty \alpha_i \int_X f \, d(q_i^* \mu) =  \sum_{i=1}^\infty \int_Y  \alpha_i  \theta_i(f) \,  d\mu 
= \int_Y  \sum_{i=1}^\infty  \alpha_i  \theta_i(f) \,  d\mu,$$
so $ T= \sum_{i=1}^\infty \alpha_i \theta_i$ is the dual operator for $S$. $T$ is nonlinear, since $\theta_i$ are nonlinear. 
\end{proof}

\begin{remark}
1) Theorem \ref{ITcontin} generalizes the fact that  $q^*$  on finite (deficient) topological measures is a Markov-Feller operator, 
whose nonlinear dual operator is $\theta$ (see \cite[Sect. 5]{Butler:QLM}). \\
2) We would like to point out that $q^* \mu$, in general, is not a measure even if $ \mu$ is (see, for instance, examples in \cite[Sect. 2]{Butler:QLM}). 
\end{remark}

\begin{definition}
A (deficient) topological measure is called representable if it is in the closed convex hull of simple (deficient) topological measures. 
\end{definition}

\begin{theorem} \label{ITnet}
Suppose $\mu$ is a representable (deficient) topological measure on a compact space $X$.
Then there is a net of Markov operators $(S_t)$  
such that each $S_t:  \mathbf{TM}_1(X) \longrightarrow \mathbf{TM}_1(X)$ (resp., $S_t:  \mathbf{DTM}_1(X) \longrightarrow \mathbf{DTM}_1(X)$)
has an invariant (deficient) topological measure $\mu_t$, and
$ \mu_t \Longrightarrow \mu$. 
\end{theorem}

\begin{proof}
We  can use the same argument as in \cite[Pr. 5.12]{Pedersen}  where it is applied for finitely many topological measures on a compact space. 
Since $\mu$ is representable, there is a net $\mu_t  \Longrightarrow \mu$, where each $\mu_t$ is a convex linear combination of simple 
(deficient) topological measures. Write $\mu_t = \sum_{i=1}^n \alpha_i \mu_i$, where  $\mu_i$'s are simple. 
Let $\mathcal{M}$ be $\mathbf{TM}_1(X)$ or $\mathbf{DTM}_1(X)$.   
Define $S_t: \mathcal{M} \longrightarrow \mathcal{M}$ as 
$S_t(\nu) = \sum_{i=1}^n \alpha_i q_i^*(\nu)$, 
where each  $q_i$ is a (d-) image transformation on $X$ 
given by Example \ref{consQst} using $\mu_i$, so $q^* \nu = \mu_i$ for any $ \nu \in \mathcal{M}$.
Then  each $S_t$ is a Markov operator by Theorem \ref{ITcontin},  
$S_t(\mu_t) =\sum_{i=1}^n \alpha_i  q_i^*(\mu_t) =   \sum_{i=1}^n \alpha_i \mu_i = \mu_t \Longrightarrow \mu$,
and $\mu_t$ is an invariant (deficient) topological measure for $S_t$. 
\end{proof}

\begin{theorem} \label{FractIT}
Suppose $X$ is a locally compact bounded metric space. 
Let $\mathcal{M}$ be a uniformly bounded in variation family of (deficient) topological measures on $X$.
Let $\{ X, q_i^*, \alpha_i \} $ be a (d-) image transformation system with contractivity factor  
$s <1, \alpha_i \ge 0$ and $\sum_{i=1}^\infty \alpha_i \le 1$.
The  Markov operator $S$ associated with $\{ X, q_i^*, \alpha_i \} $ 
is a contraction on $(\mathcal{M}, d_{KR})$, and has the unique 
fixed point $\mu_0 \in \mathcal{M}$. If $\mathcal{M}$ is  $\mathbf{DTM}_1(X)$ or $ \mathbf{TM}_1(X)$ then $\mu_0$ is representable. 
\end{theorem}

\begin{proof}
Let  $S$ be a Markov operator associated with the ITS  $(X,  q_i^*, \alpha_i)$.
As in the proof of Theorem \ref{ITcontin}, $S :  \mathcal{M} \rightarrow \mathcal{M}$. 
Let $ \theta_i$ be a (conic) quasi-homomorphism corresponding to $q_i$. Let $f  \in  Lip_1(X,d) \cap C_c(X)$ 
(respectively, $f  \in  Lip_1(X,d) \cap C_c^+(X)$).
By Theorem \ref{sThe}\ref{E},  $\frac{1}{s_i} \theta_i(f)$ is in $  Lip_1(X,d) \cap C_c(X)$ (respectively, $  Lip_1(X,d) \cap C_c^+(X)$).
Applying  Corollary \ref{inftSuDTM} and Theorem \ref{ITqh}\ref{ITqh5} we have:
\begin{align*}
| \int_X f \, d(S(\mu))  & -  \int_X f \, d(S(\nu)) | = |\sum_{i=1}^\infty \alpha_i \int_X f \, d(q_i^* \mu) - \sum_{i=1}^\infty \alpha_i \int_X f \, d(q_i^* \nu) | \\
& = |  \sum_{i=1}^\infty \alpha_i   \int_X \theta_i(f) \, d \mu - \alpha_i  \int_X \theta_i(f) \, d \nu |  
\le s \sum_{i=1}^\infty \alpha_i   | \int_X  \frac {1}{s_i} \theta_i(f) \, d \mu - \int_X   \frac {1}{s_i} \theta_i(f) \, d \nu |   \\
& \le s \sum_{i=1}^\infty \alpha_i d_{KR} (\mu, \nu)  \le  s d_{KR} (\mu, \nu).
\end{align*}

It follows that $d_{KR} (S(\mu), S(\nu)) \le  s d_{KR} (\mu, \nu)$, so $S$ is a contraction.  
By Theorem \ref{completeMS} $(\mathcal{M}, d_{KR})$ is a complete metric space, so there is a unique fixed point $\mu_0 \in \mathcal{M}$. 

Take normalized (deficient) topological measures as  $\mathcal{M}$. The set of representable (deficient) topological measures is compact, 
and it is invariant under $S$. Hence, iterations of a representable (deficient) topological measure
by $S$ will converge to a representable (deficient) topological measure, which is also the unique fixed point  $\mu_0$. 
\end{proof}

Let $P(X)$ denote the set of probability measures on $X$.

\begin{theorem} \label{Barnsley}
Suppose $(X, d)$ is a compact metric space, $u_i :X \longrightarrow X$ are contraction maps for $ i =1,2, \ldots$ with contractivity factors
$s_i$, each $ s_i \le s <1$, $ \alpha_i \ge 0$,  
and $ \sum_{i=1}^\infty \alpha_i \le 1$.  Consider  Markov operator 
$M: P(X) \longrightarrow P(X)$ given by 
$ M(\nu) = \sum_{i=1}^\infty \alpha_i \  \nu \circ u_i^{-1}.$
Then $M$ is a contraction with contactivity factor $ s$ with respect to the KR metric on $P(X)$, i.e.  
$d_{KR}(M(\mu), M(\nu)) \le s d_{KR}(\mu, \nu).$
In particular, there is a unique measure $m_0 \in P(X)$ such that $ M(m_0) = m_0$.
In the case of finitely many contractions $u_1, \ldots, u_n$ and  $ \sum_{i=1}^n \alpha_i \le 1, \   \alpha_i \ge 0$ 
the support of  $m_0$ is the attractor of the IFS $ \{X, u_1, \ldots, u_n\}$. 
\end{theorem}

\begin{proof}
With the help of Corollary \ref{inftSuDTM} the proof is a minor modification of \cite[Ch. IX, Th. 6.1]{Barnsley} which is done for finitely many contractions.
Alternatively, note that Theorem \ref{Barnsley} is a particular case of Theorem \ref{FractIT} using $q_i = u_i^{-1}$ and $P(X)$ in place of $ \mathcal M$. 
The last statement is well-known result, see, for instance \cite[Ch. IX, Th. 6.2]{Barnsley} or \cite[Th. 2.8]{Falconer}. 
 \end{proof}

\begin{definition}
For a representable deficient topological measure $\mu$ we define $D^\mu   = \bigcap \{ K^{\flat} \subseteq X^{\flat}: \mu(K) = 1, \, K \in \mathscr{K}(X)\}$.
Similarly,  $D^\mu   = \bigcap \{ K^{\sharp} \subseteq X^{\sharp}: \mu(K) = 1, \, K \in \mathscr{K}(X)\}$ for a representable topological measure $\mu$.
\end{definition}

Since $X^\flat$ and $X^\sharp$ are compact, by finite intersection property  $D^\mu$ is a nonempty compact set. 

\begin{remark}
In \cite{Aarnes:ITfirst} the set $D^\mu$  is defined using topological measures on a compact space, and even though it is in $X^\sharp$, it  
is (somewhat confusingly) called the support of $\mu$.
For a compact space $X$ the notion of a support  of a topological measure $ \mu$,  $ supp \,  \mu$,  is defined in \cite{Laberge}.
It is a closed subset of $X$, and it is easy to see that  $(supp \  \mu) ^\sharp = D^\mu$.
\end{remark}

\begin{remark} \label{IFSord}
Suppose  $q_1, \ldots, q_n$ are  (d-) image transformations with contractivity factors $s_i <1$. 
Each $q_i^*$ leaves $X^{\flat}$ (respectively, $X^\sharp$) invariant 
and a simplified argument from Theorem \ref{FractIT} shows it is a contraction with respect to $d_{KR}$ with contractivity $s_i$. 
Then  $\{ X^{\flat}, q_1^*, \ldots, q_n^*\} $ (resp., $\{ X^{\sharp}, q_1^*, \ldots, q_n^*\} $) is an IFS in a  classical sense, see 
for instance, \cite[Ch. III, Def. 7.1]{Barnsley}. 
By \cite[Ch. III, Th. 7.1]{Barnsley}, the map 
$W(F) = \bigcup_{i=1}^n q_i^*(F)$ where $F \in \mathscr{C}  (X^{\flat})$ (resp,  $F \in \mathscr{C}  (X^{\sharp})$) 
is a contraction with contractivity factor $s = \max s_i $, 
and has a unique fixed point $A \subseteq \mathscr{C} (X^{\flat})$ (resp., $A \subseteq \mathscr{C} (X^{\sharp})$),
called the attractor of the given IFS. So $W(A) =  \bigcup_{i=1}^n q_i^*(A)  = A$.
\end{remark}

The next theorem relates an image transformation system on $X$  and the corresponding  Markov operator  
to the IFS on $X^{\sharp}$ and its Markov operator. 

\begin{theorem} \label{IFSonXandX*}
Suppose  $\mathcal{M}$ is a uniformly bounded in variation family of (deficient) topological measures on a locally compact metric space $X$, 
 $q_i$  
is a (d-) image transformation on $X$ with contraction factors $s_i \le s <1$, 
$ \alpha_i \ge 0$ for $i=1, 2, \ldots$, 
and $ \sum_{i=1}^{\infty} \alpha_i \le 1$. 
Let representable deficient topological measure $\mu_0$ be the fixed point given by Theorem \ref{FractIT} for Markov operator
$S$ associated with $ (X,  q_i^*, \alpha_i)$,  and let $D_0 = D^{\mu_0}$.
Let $m_0$ be the probability measure on $X^{\sharp}$ (resp, on $X^\flat$) which is the fixed point given by Theorem \ref{Barnsley} 
for Markov operator $M$ with $q_i^*:X^{\sharp} \rightarrow X^{\sharp}$ (resp., $q_i^*:X^{\flat} \rightarrow X^{\flat}$) as contraction maps.
Then
\begin{enumerate}[label=(\arabic*),ref=(\arabic*)]
\item
$S(m \circ \Lambda) = (M(m)) \circ \Lambda$ for any probability measure $m$ on $X^{\sharp}$ (resp., on $X^\flat$).
\item
$\mu_0 = m_0 \circ \Lambda$.
\item \label{qido}
If   $ \sum_{i=1}^\infty \alpha_i = 1$ then $q_i^*(D_0) \subseteq D_0$ for each $i$.
\end{enumerate}
For finitely many (d-) image transformations $q_1, \ldots, q_n$ and $ \sum_{i=1}^n \alpha_i = 1$,
 let transformation $W$ and its attractor $A$ be as in Remark \ref{IFSord}.  Then
\begin{enumerate}[label=(\roman*),ref=(\roman*)]
\item \label{wdodo}
$W(D_0) \subseteq D_0$.
\item
$A = supp \  m_0   \subseteq D_0$.
\item 
If $\mu_0$ is a simple topological measure, then $A = supp \  m_0 = D_0 =  \{ \mu_0\}$, so $m_0$ is the point mass at $ \mu_0$.
\end{enumerate}
\end{theorem}

\begin{proof}
(1) For a probability measure $m$ on $X^{\sharp}$ consider a representable topological measure $\mu =  \Lambda^* m = m \circ  \Lambda$ on $X$.
Let $C \in \mathscr{K}(X)$. 
By Theorem \ref{ITqh}
$(q_i(C))^{\sharp} =(q_i^*)^{-1} (C^{\sharp})$. Then 
\begin{align*}
 S(m \circ \Lambda) (C) &= S(\mu) (C) = \sum_{i=1}^\infty \alpha_i q_i^* \mu (C) = \sum_{i=1}^\infty \alpha_i \mu(q_i (C)) \\
 &= \sum_{i=1}^{\infty} \alpha_i m ((q_i (C))^{\sharp}) = \sum_{i=1}^{\infty} \alpha_i m( (q_i^*)^{-1} (C^{\sharp})) = M(m) (C^{\sharp}) =  
 (M(m)) \circ \Lambda (C).
\end{align*}
Thus, $S(m \circ \Lambda) = (M(m)) \circ \Lambda$. \\
(2) Since $m_0$ is the fixed point for $M$, by the previous part we have $S(m_0 \circ \Lambda) = (M(m_0)) \circ \Lambda = m_0 \circ \Lambda$.
Thus, $m_0 \circ \Lambda$ is the fixed point for $S$, so $m_0 \circ \Lambda = \mu_0$. \\
(3) Suppose  $ \sum_{i=1}^\infty \alpha_i = 1$. For any $A \in \mathscr{O}(X) \cup \mathscr{K}(X)$ we have $ \mu_0(A) = S(\mu_0)(A) = \sum_{i=1}^\infty \alpha_i \mu_0(q_i(A))$,
and so $ \mu_0(A) =1$ iff $\mu_0(q_i(A)) =1 $ for all $i$.  
If $ \nu \in D_0$ then $q_i^* \nu(K) = \nu(q_i(K)) = 1$ for any compact $K$ such that $\mu_0(K) = \mu_0(q_i(K)) =1 $.
Thus, $q_i^* \nu \in D_0$, so $q_i^*(D_0) \subseteq D_0$. \\
(i) Follows from part \ref{qido} and Remark \ref{IFSord}. \\
(ii) The equality  $A = supp \  m_0$ is given by \cite[Ch. IX, Th. 6.1,Th. 6.2]{Barnsley}.
Since $ A = \lim_{n \rightarrow \infty} W^n(C) $ for any nonempty closed $C$ in $X^{\sharp}$ (\cite[Ch. III,  Th. 7.1]{Barnsley}), 
from part \ref{wdodo}  we see that $A \subseteq D_0$. \\
(iii) Let $ \mu_0$ be simple. For $ \nu \in D_0$,   $\mu_0(K) = 1$ implies $\nu(K) = 1$, and  
$\mu_0(K) = 0 $ implies there is $C \subseteq X \setminus K$ with $\mu_0(C) = 1$, so $ \nu(C) = 1$, and $\nu(K) = 0$. Thus, $\mu_0 = \nu$ on compact sets, i.e.
$ \nu =\mu_0$.
Then $D_0 = \{ \mu_0\}$. The $supp \  m_0$, a nonempty closed set, must be $D_0$, so $m_0$ is the point mass at $ \mu_0$.
\end{proof}

\begin{remark}
In general, $A$ is a proper subset of $D_0$ (see \cite[Rem. 5.7]{Aarnes:ITfirst}).
\end{remark}

\begin{remark}
For topological measures and a compact space $X$, 
parts \ref{E} - \ref{A} of Theorem \ref{sThe} first appeared in 
\cite[Pr. 5.2]{Aarnes:ITfirst}  and then in \cite[Pr. 5.8]{Pedersen}, \cite[Pr. 36]{AarnesJohansenRustad}.  
In Theorem \ref{FractIT} we generalize a similar result for compact space $X$, topological measures, and image transformation of the form
$S(\mu) = \sum_{i=1}^n \alpha_i q_i^* (\mu),\alpha_i \ge 0, \sum_{i=1}^n \alpha_i = 1$, see
\cite[Pr. 5.3]{Aarnes:ITfirst},  \cite[Pr. 5.13]{Pedersen}, \cite[Pr. 44]{AarnesJohansenRustad}. 
The first four statements in Theorem \ref{IFSonXandX*} 
for topological measures generalize \cite[Th. 5.5]{Aarnes:ITfirst} proved for a compact space. 
\end{remark}

\section{Signed topological measures and covering dimension} \label{SectDimSTM}

There are several results linking topological measures and dimension theory. First, R. Wheeler in \cite{Wheeler}  showed that 
any topological measure on a normal space $X$ with large inductive dimension $ \text{Ind } X \le 1$ is a measure, 
and if $\dim X \le 1$, then any simple topological measure is a measure.  
This result was later improved by D. Shakmatov in an unpublished note, where $ \text{Ind } X \le 1$ was replaced by a weaker condition  $ \dim X \le 1$.
The theorem that any signed topological measure on a compact space $X$ with $ \dim X \le 1$ is a signed measure 
was first proven by D. Grubb in \cite{Grubb:SignedqmDimTheory}, and later by M. Svistula  in \cite{Svistula:Signed}.  
All these results were obtained by different methods. 
We shall show that on a locally compact space $X$ with $ \dim X \le 1$ 
any signed topological measure of finite norm is a signed Radon measure. 
This result generalizes all previous results involving dimension theory, and 
is very important on its own. We will also need it in the next section.

\begin{definition} \label{STMLC}
A signed topological measure on a locally compact space $X$ is a set function
$\mu: \mathscr{O}(X) \cup \mathscr{C}(X) \rightarrow [-\infty, \infty]$  that assumes at most one of $\infty, 
-\infty$ and satisfies the following conditions:
\begin{enumerate}[label=(STM\arabic*),ref=(STM\arabic*)]
\item \label{STM1} 
if $A,B, A \sqcup B \in \mathscr{K}(X) \cup \mathscr{O}(X) $ then
$\mu(A\sqcup B)=\mu(A)+\mu(B);$
\item \label{STM2}  
$\mu(U)=\lim\{\mu(K):K \in \mathscr{K}(X), \  K \subseteq U\}
$ for $U\in\mathscr{O}(X)$;
\item \label{STM3}
$\mu(F)=\lim\{\mu(U):U \in \mathscr{O}(X), \ F \subseteq U\}$ for  $F \in \mathscr{C}(X)$.
\end{enumerate}
If in \ref{STM1} sets $A$ and $B$ are compact, then $\mu$ is called a signed deficient topological measure. 
\end{definition} 

Note that a signed (deficient) topological measure which is identically 0 on compact sets (or on open precompact sets)  is equal to 0.

\begin{definition} \label{laplu}
Given signed set function $\lambda: \mathscr{K}(X)  \longrightarrow [-\infty, \infty] $ which assumes at most one of $ \infty, -\infty$
we define its total variation $| \lambda|$  
as follows: 
$$|\lambda| (U) = \sup \{ \sum_{i=1}^n |\lambda(K_i)| : \  \bigsqcup_{i=1}^n K_i \subseteq U, \ K_i \subseteq \mathscr{K}(X),  \, n \in \mathbb{N} \} \text{   for an open  } U, $$
$$ |\lambda| (F)  = \inf\{ |\lambda| (U) : \ F \subseteq U, \ U \in \mathscr{O}(X)\} \text{   for a closed  } F. $$
\end{definition}

\begin{remark} \label{totvardtm}
From \cite[Th. 3.8]{Butler:DTMLC} it follows that if $ \mu$ is a signed (deficient) topological measure, then $ | \mu |$ is a deficient topological measure.
From \cite[L. 2.3]{Butler:DTMLC} we have $ |\mu(A)| \le |\mu| (A)$ for a closed or open set $A$. 
\end{remark}

\begin{definition} \label{SDTMnorDe}
We define  $\| \mu \| = \sup \{ | \mu(K)|  :  K \in  \mathscr{K}(X) \} $ for a signed deficient topological measure $\mu$.
\end{definition}

\noindent
For more information about $ | \mu |$ and $ \| \mu \|$ see \cite[Sect. 2]{Butler:STMLC}. We have: $\| \mu \| < \infty$ iff $ | \mu| (X)  < \infty$.
If $\mu$ is a deficient topological measure then $ \| \mu \| = \mu(X)$. 
  
\begin{definition} \label{properSDTM}
A signed deficient topological measure $\mu$  is called proper if  deficient topological measure  $ | \mu |$ is proper, i.e.
from $m \le |\mu| $, where $m$ is a Radon measure, it follows that $m = 0$.
\end{definition}

\begin{definition} \label{dimX}
For a topological space $X$ its covering dimension $\dim X \le 1$ if for every finite open cover $\{ U_1, \ldots, U_n\}$ of $X$ there is an open cover
$\{ V_1, \ldots, V_n\}$ of $X$ such that $ V_i \subseteq U_i$ for $ i=1, \ldots, n$ and $V_i \cap V_j \cap V_k = \emptyset$ for any triple of distinct $i, j, k$.
\end{definition}    

\begin{theorem} \label{dimSTM}
Let $X$ be a locally compact space with the covering dimension $\dim X \le 1$. 
Suppose $ \mu$ is a proper signed topological measures on $X$, $ | \mu |(X) < \infty$. Then $ \mu = 0$. 
\end{theorem}

\begin{proof}
Suppose $\mu$ is a finite proper signed topological measure on $X$,  $\dim X \le 1$, $ | \mu |(X) < \infty$. 
We shall show that $ \mu = 0$, following a technique from \cite[Pr. 12]{Svistula:Signed}.
Since $ | \mu| $ is proper, by \cite[Th. 4.4]{Butler:Decomp}  for $\epsilon >0$ choose an open cover  $\{U_1, \ldots, U_n\}$ of $X$ such that
$ \sum_{i=1}^n |\mu | (U_i) < \epsilon$.
Let $W$ be an arbitrary precompact open subset of $X$. By \cite[L. 25]{Butler:STMLC} there is a compact
$C \subseteq W$ such that $ |\mu(A)| < \epsilon$ for any compact or open $A \subseteq W \setminus C$. Let $ E = X \setminus C$. 
For an open cover $\{ U_i \cap W, U_i \cap E, i =1, \ldots, n\}$ of $X$, choose a finite open cover $\{ V_i, i \in I \}$  according to assumption
$ \dim X \le 1$. So each $V_i$ is contained in at least one of $U_i \cap W \subseteq W$ or $U_i \cap E \subseteq E$, and 
$$ \sum_{i \in I} | \mu | (V_i) \le \sum_{i =1}^n |\mu| (U_i \cap U) +  \sum_{i =1}^n |\mu| (U_i \cap E)   \le  2  \sum_{i =1}^n |\mu| (U_i) < 2  \epsilon.$$
Consider sets $ C_i =  V_i \setminus \bigcup_{j \in I, j \ne i} V_j = X \setminus \bigcup_{j \in I, j \ne i} V_j$ and  $V_{kl} =V_k \cap V_l$, where $i, j, k,l \in I$.
Since $\dim X \le1$, these sets are pairwise disjoint; their disjoint union is $X$. Note also that
$$ \sum_{ k \neq l} | \mu| (V_{kl}) \le \sum_{i \in I}  \sum_{j \neq i} |\mu| ( V_{ij}) \le  \sum_{i \in I} |\mu| (V_i).$$ 
Let $I_1 = \{ i \in I: C_i \not\subseteq W \}$. 
For $ i \in I_1$, $V_ i \not\subseteq W$, so $V_ i \subseteq E$, so $ C_ i \subseteq E = X \setminus C$.  
Set $U = W \setminus \bigsqcup_{i \in I_1} C_i$. Then $ C \subseteq U \subseteq W$. 
Since $U \setminus C \subseteq W \setminus C$,  we have
$| \mu(U\setminus C )|  = | \mu(U) - \mu(C)| < \epsilon$, and $| \mu(W) - \mu(C)| < \epsilon$, so $ |\mu(U) - \mu(W) | < 2 \epsilon$.  
Let $I_2 = I \setminus I_1$. For $ i \in I_2$ the set $C_i \subseteq U \subseteq W$, so $C_i$ is compact.
We have $U = \bigsqcup_{i \in I_2} C_i \sqcup \bigsqcup_{k \neq l} (V_{kl} \cap W)$, 
so $ \mu(U) = \sum_{i \in I_2} \mu(C_i) + \sum_{k \neq l} \mu(V_{kl} \cap U)$. 
Now:
\begin{align*}
|\mu(W)| - 2 \epsilon & \le |\mu(U) |  \le \sum_{i \in I_2} |\mu(C_i)| + \sum_{k \neq l} | \mu(V_{kl} \cap W) | \le 
 \sum_{i \in I_2} |\mu| (C_i) + \sum_{k \neq l} | \mu | (V_{kl} \cap W)   \\
& \le  \sum_{i \in I_2} |\mu| (C_i) + \sum_{k \neq l} | \mu | (V_{kl})   \le 2 \sum_{i \in I} |\mu|(V_i) < 4 \epsilon,
\end{align*}
and $ |\mu(W)| < 6 \epsilon$. It follows that $ \mu(W) = 0$ for any precompact open set $W$. Thus, $ \mu = 0$. 
\end{proof}

\begin{theorem} \label{dimTh}
Let $X$ be a locally compact space with the covering dimension $\dim X \le 1$. 
Suppose $ \nu$ is a (signed) topological measures on $X$, $ | \nu |(X) < \infty$. Then $ \nu$ is a (signed) Radon measure.
\end{theorem}

\begin{proof}
$\| \nu \| < \infty$, and by \cite[Th. 5.3]{Butler:Decomp} we may write $ \nu = m + \mu$, where $m$ is a signed Radon measure, 
$ \mu$ is a proper signed topological measure. By  Theorem \ref{dimSTM} $ \mu = 0$. 
\end{proof}

\begin{remark}
Theorem \ref{dimTh} is not true for signed deficient topological measures. On $X = \mathbb{R}$ (so $\dim X \le 1$) 
there are finite deficient topological measures 
that are not measures, see \cite[Sect. 6]{Butler:DTMLC}.
\end{remark}  

\section{Sample Median} \label{SectSampleMed}

The idea to define a median in multidimensional spaces started well over 100 years ago. In \cite{Hayford} J. Hayford, regarding the problem of 
finding the center of the US population, considered a couple of ways to define the median in two-dimensional spaces 
and discussed their pluses and minuses in the light of some natural equivariance properties. 
Since then there were many proposed definitions of a higher 
dimensional medians, see, for instance, \cite{Small} for a survey of various concepts. 
One desirable feature of a multidimensional median is the equivariance with respect to some natural transformations, for example, 
projections onto coordinate axes, monotone maps, etc.   
Following ideas of  A. Rustad in \cite{AlfMedian},  we define a generalized distribution of the sample median (g.d.s.m.) 
for finitely many continuous proper maps $T_i : Y \longrightarrow X$. 
Here  $X$ and $Y$ are locally compact, and we do not require the spaces to be metric.
We then show that the g.d.s.m. and the inverse on the sample median are equivariant under a wide collection of maps, so called solid variables, 
which include projections, monotone maps, homeomorphisms, etc.     
To prove the results of this section we again use image transformations and their adjoints. 

We denote  by $|S|$ the cardinality of a finite set $S$, and by $I_n$ the set $ \{1, \dots, n\}$. 

\begin{definition}  \label{SampleMed}
Let $X$ and $Y$ be locally compact spaces, and $P$ be a deficient topological measure (in particular, a probability measure) on $Y$. 
If $T_i : Y \longrightarrow X$ for $i=1, \ldots, 2n-1$ 
are continuous proper maps, we define the generalized distribution of the sample median (g.d.s.m.) of the maps $ \{T_i\}$ to be a set function 
$ \mu_{\{T_i\}}: \mathscr{A}_{s}^{*}(X)  \longrightarrow \mathbb{R}$ with
$ \mu_{\{T_i\}}(A) = P(\{ y:  | \{i: T_i(y) \in A \}| \ge n\})$. 
\end{definition}

When $P$ is a probability measure on $Y$,  
$ \mu_{\{T_i\}}(A)$  is the probability of over half of the variables $T_i$ being in $A$. 

\begin{theorem} \label{SMedExt}
Suppose $X$ is a q-space or a noncompact q-space.
Suppose $Y$ is locally compact,  and $T_i : Y \longrightarrow X$ for $i=1, \ldots, 2n-1$ 
are continuous proper maps. Then the map $q$ defined on  $\mathscr{A}_{s}^{*}(X)$ by 
$$ q(A) = q_{\{T_i\}}(A) = \bigcup_{ \{S \in I_{2n-1}: |S| \ge n \} } \left[ \bigcap_{i \in S} T_i^{-1}(A) \right] $$
extends to an image transformation from $X$ to $Y$.
If $P$ is a deficient topological measure on $Y$,
then the  g.d.s.m. of the maps $ \{T_i\}$ 
extends to a topological measure $ \mu_{\{T_i\}} = q^*P = P \circ q$  on $X$.
\end{theorem}

\begin{proof} 
Suppose $X$ is a noncompact  q-space.
Using Theorem \ref{ITsolid} and Remark \ref{easysolpar}
we shall show that $q$ in the statement of the theorem extends to an image transformation from $X$ to $Y$. 

Note that $q$ is monotone,  $q(K)$ is compact for $K \in \mathscr{K}_{s}(X)$, and  $q(U)$ is open for $U \in  \mathscr{O}_{s}^*(X)$. 
Suppose  $\bigsqcup_{j=1}^k C_j \subseteq C,  \  C, C_j \in \mathscr{K}_{s}(X)$. Let $y \in q(C_j)$ for some $j$.
If  $  T_i(y) \in C_j$ for all $i$ in some $S$, $ |S| \ge n$, then $| \{ i:  T_i (y) \in C_l \}| < n$ for any set $C_l, l \neq j$. 
Thus, the sets $q(C_1), \ldots, q(C_k)$  are disjoint, 
and we obtainTheorem \ref{ITsolid}\ref{ITsol1}.  
To show Theorem \ref{ITsolid}\ref{ITsol2}, let $ y \in q(U)$, $U \in \mathscr{O}_{s}^{*}(X)$, so
$ D = \bigcup_{i \in S} T_i(y) \subseteq U$ for some $S$ with  $|S|  \ge n$. 
By \cite[Th. 3.10 (1)]{Butler:TMLCconstr} pick a compact solid set $K$ such that $ D \subseteq K \subseteq U$, 
and observe that $y \subseteq q(K) \subseteq q(U)$.
Theorem \ref{ITsolid}\ref{ITsol3} follows similarly from \cite[Th. 3.10 (2)]{Butler:TMLCconstr}, and
Theorem \ref{ITsolid}\ref{Qsolidparti} holds automatically by Remark \ref{easysolpar}. 

If $X$ is a q-space, the argument is similar, noting that given a solid partition $X = A \sqcup (X \setminus A)$, 
for any $y \in q(X)$ over half of all $ T_i(y)$ is contained in either $A$ or $X \setminus A$, so $q(X) = q(A) \sqcup q(X \setminus A)$. 
By Remark \ref{easysolpar}
this gives Theorem \ref{ITsolid}\ref{Qsolidparti}. 

The last assertion follows from  \cite[Th. 10.7]{Butler:TMLCconstr}, 
since  $q^* P = P \circ q$  is a topological measure, and it coincides with  $ \mu_{\{T_i\}}$ on $\mathscr{A}_{s}^{*}(X)$. 
\end{proof}

\begin{definition} 
A function $f :X_1 \longrightarrow X_2$ is called a solid variable if it is continuous and for a solid (open or closed) set $A \subseteq X_2$,  
$f^{-1} (A)$ is a solid set in $X_1$.
\end{definition}

\begin{remark}
If $f :X_1 \longrightarrow X_2$ is a solid variable which is also proper, then $f^{-1} (\mathscr{K}_{s}(X_2))  \subseteq \mathscr{K}_{s}(X_1)$.
A composition of solid variables is a solid variable. 
Solid variables on compact spaces are studied in \cite{AarnesJohansenRustad}, \cite{OrjanAlf:HomSimpleTM},   \cite{AlfImTrans}, \cite{AlfMedian}.
\end{remark} 

\begin{example} \label{solidVar}
A homeomorphism $h: X \longrightarrow Y$ is a solid variable, here $X,Y$ are compact or LC noncompact. 
A  continuous onto map $ h: X \longrightarrow Y$, $X, Y$ are compact, for which inverse images of points are connected, is a solid variable. 
For  $X = Y = \mathbb{R}$, monotone maps are solid variables. 
Projections of the closed ball in $\mathbb{R}^n$  centered at the origin onto lines in $\mathbb{R}^n$ (in particular, onto coordinate axis)  are solid variables. 
Examples of proper solid variables when $X = Y = \mathbb{R}^n, \ n \ge 2$ include rotations, 
the symmetry with respect to  a line, the symmetry with respect to a point,  translations,  
a function $f_t$ given by $f_t(0) = 0, f_t(x) = tx$ for $x \neq 0$, where $ t \in \mathbb{R} \setminus \{0\}$.  
\end{example}
 
\begin{theorem} \label{ITfinva}
Suppose $X_i$ ($i=1,2$) is a q-space or a noncompact q-space.
Suppose $Y$ is locally compact,  $T_i: Y \longrightarrow X_1$ for $ i =1, \ldots, 2n-1$ are proper continuous maps, 
and $f: X_1  \longrightarrow X_2$ is a proper solid variable. 
Let $q_{\{T_i\}}$ be the image transformation given by Theorem \ref{SMedExt}.
Then  $q_{\{T_i\}}  \circ f^{-1}$ is an image transformation, and
$q_{\{T_i\}}  \circ f^{-1} = q_{\{f \circ T_i\}}$,  hence for their adjoints
$ (f^{-1}) ^* \circ q_{\{T_i\}} ^* = q_{\{f \circ T_i\}}^*. $ 
\end{theorem}

\begin{proof}
By Example \ref{invfunIT}, $f^{-1}$ is an image transformation from $X_2$ to $X_1$, so  $q_{\{T_i\}}  \circ f^{-1}$ is an image transformation.
Let $K$ be a compact solid set in $X_2$. Then
\begin{align*}
 y \in q_{\{T_i\}}  (f^{-1} (K))  & \Longleftrightarrow | \{ i: T_i(y) \in f^{-1}(K) \} | \ge n 
 \Longleftrightarrow | \{ i: f(T_i(y)) \in K  \}|  \ge n \\
 & \Longleftrightarrow y \in q_{\{f \circ T_i\}} (K),
\end{align*}
which shows that  $q_{\{T_i\}}  \circ f^{-1} = q_{\{f \circ T_i\}}$ as image transformations.
\end{proof}

\begin{definition}  \label{SampleMedE}
Let $X$ and $Y$ be locally compact spaces, and $P$ be a deficient topological measure on $Y$.
For an even number of maps $T_i : Y \longrightarrow X$ for $i=1, \ldots, 2n$, we define 
the g.d.s.m. 
$ \mu_{\{T_i\}} = \frac{1}{2n} \sum_{j=1}^{2n}  \mu_{\{T_i\}}^j$, where 
$ \mu_{\{T_i\}}^j$ is the g.d.s.m. for $2n+1$ maps $T_1, \ldots, T_{2n}, T_j$.
\end{definition}

\begin{definition}
Given maps $T_i: Y \longrightarrow \mathbb{R}$, $ i=1, \ldots, n$, the kth order statistic 
is the function $T_{(k)} : Y  \longrightarrow \mathbb{R}$ that assigns each $y$ the kth smallest value
of $\{ T_i(y)\}$. For  the maps $T_1, \ldots, T_{2n-1}$, its nth order statistic $T_{(n)}$ is called the sample median.
\end{definition} 

\begin{theorem} \label{SaMeProbD}
Suppose $Y$ is a locally compact space with probability $P$ on $Y$,  and  $\{T_i : Y \longrightarrow \mathbb{R} \}$ is a finite collection of bounded continuous proper maps. 
\begin{enumerate}[leftmargin=0.25in, label=(\roman*),ref=(\roman*)]
\item \label{oddSaMe}
For an odd number of maps $T_1, \ldots, T_{2n -1}$ the g.d.s.m. 
is the probability distribution of the sample median. i.e. $ \mu_{\{T_i\}} = P \circ T_{(n)}^{-1}$. Thus, the image transformation $q_{\{T_i\}} = T_{(n)}^{-1}$.
\item \label{evenSaMe}
For an even number of maps $T_1, \ldots, T_{2n} $ the g.d.s.m. $ \mu_{\{T_i\}} = \frac12 P \circ T_{(n)}^{-1} + \frac12 P \circ T_{(n+1)}^{-1}$.
\item \label{evenSaMeAlf}
 For an even number of maps $T_1, \ldots, T_{2n} $  the g.d.s.m. $ \mu_{\{T_i\}} = \frac{1}{2n} \sum_{j=1}^{2n}  \mu_{\{T_i\}}^{-j}$, where 
$ \mu_{\{T_i\}}^{-j}$ is the g.d.s.m. for  $2n-1$ maps $\{T_i,   i \in I_{2n}, i \ne j  \}$.
\end{enumerate}
\end{theorem}

\begin{proof}
We have $T_i : Y \longrightarrow [c,d] \subseteq \mathbb{R}$ for all $i$,  
and by Theorem \ref{SMedExt} $\mu_{\{T_i\}}$ is a topological measures on $[c,d]$. 
By Theorem \ref{dimTh},  the topological measure $ \mu_{\{T_i\} }$ is actually a measure on $[c,d]$. \\
\ref{oddSaMe}. Consider an odd number of maps $T_1, \ldots, T_{2n-1}$.
Note that $T_i(y) < b$ for more than half of $T_i's$ iff $T_{(n)}(y) < b$, so for sets of the form $A = [c, b)$ or $[c, b]$ 
we have $ \mu_{\{T_i\}}(A) = P \circ T_{(n)}^{-1}(A)$. 
Since measures $ \mu_{\{T_i\}}$ and $ P \circ T_{(n)}^{-1}$ coincide on these intervals,  we have $\mu_{\{T_i\}} =P \circ T_{(n)}^{-1}$. 
Note that this holds for any $P$.  To show that  $q_{\{T_i\}} = T_{(n)}^{-1}$, assume to the contrary that $q_{\{T_i\}} (K) \neq T_{(n)}^{-1} (K)$
for some compact $K$. Say, $x  \in T_{(n)}^{-1} (K) \setminus q_{\{T_i\}} (K) $. Taking $P = \delta_x$, 
by Theorem \ref{SMedExt} $\mu_{\{T_i\}} =P \circ q_{\{T_i\}}  \neq P \circ T_{(n)}^{-1}$, a contradiction. \\ 
\ref{evenSaMe}.
Let $T_{(n)}$ be the  nth order statistic for  $T_1, \ldots, T_{2n}$.
For augmented collection $\{ T_1, \ldots, T_{2n}, T_j \}$ its $(n+1)$st order statistic is denoted by  $T_{(n+1)}^j$. 
Note that 
$T_{(n+1)}^j(y) = T_{(n)} (y)$  for 
half of $j  \in \{1, \ldots 2n \}$, and it is $T_{(n+1)}^j(y) =T_{(n+1)} (y)$ for the other half of $j  \in \{1, \ldots 2n \}$. 
The statement follows from Definition \ref{SampleMedE} and part \ref{oddSaMe}. \\
\ref{evenSaMeAlf}.
We can show that 
$\frac{1}{2n} \sum_{j=1}^{2n}  \mu_{\{T_i\}}^{-j} = \frac12 P \circ T_{(n)}^{-1} + \frac12 P \circ T_{(n+1)}^{-1}$
by demonstrating that these two measures coincide for the sets of the form $[c, b)$. One can use the argument from \cite[Prop.23]{AlfMedian}.
Note that an argument similar to one in \cite[Prop.23]{AlfMedian}  can also prove part \ref{evenSaMe}. 
\end{proof} 

\begin{remark}
In \cite{AlfMedian} $ \mu_{\{T_i\}}(A)$ is called the sample median, which is confusing, given Theorem \ref{SaMeProbD}\ref{oddSaMe}.
This what lead us to use our terminology for  $ \mu_{\{T_i\}}(A)$. In  \cite{AlfMedian}   $\mu_{\{T_i\}}$ was defined using a compact metric space $X$, and  
probability $P$ on a compact space $Y$, with a slightly different definition of an image transformation.   
\end{remark} 

\begin{remark}
Let $ x_1, x_2, \ldots, x_{2n-1}$ be ordered real numbers, so their median is $x_n$. We shall show that  the g.d.s.m.  is $\delta_{x_n}$.
Let $Y$ be any compact space, $P$ be any probability measure on it, and $X = [x_1, x_n]$. Define $T_i: Y \longrightarrow X$ to be constant maps
$T_i \equiv x_i$ for $i=1, \ldots, 2n-1$, so $T_i$ are continuous proper maps. Let $U_1 = [x_1, x_n)$ and $U_2 =(x_n, x_{2n-1]}$. Then 
$U_1, U_2$ are open solid sets, and $\mu_{\{ T_i \} }(U_1) = \mu_{\{ T_i \} }(U_2) = 0$. Then $\mu_{\{ T_i \} }(\{x_n\}) = 1$, 
and by Remark \ref{tausm},  $ \mu_{\{ T_i \} } = \delta_{x_n}$.

For $x_1, x_2, \ldots, x_{2n}$ the median is $\frac12( x_n + x_{n+1})$, 
and the latter is equal to $\frac{1}{2n}(med_1 + med_2 + \ldots  + med_{2n})$, where
$ med_j$ denotes the median of $\{x_1, \ldots, x_{2n}, x_j \}$. Then it is easy to see that again $\mu_{\{ T_i \} } = \frac12  (\delta_{x_n} + \delta_{x_{n+1}})$.
\end{remark}

\begin{remark}
It is interesting to note that if $f$ is a solid variable,  and $\mu_1, \mu_2, \mu_3$ are simple topological measures with the corresponding point masses
$\delta_{x_1}, \delta_{x_2}, \delta_{x_3} $ on $\mathbb{R}$ (see \cite[L. 6.11]{Butler:ReprDTM}),
then the median of $\{ x_1, x_2, x_3\}$ is also obtained by integrating $f$ 
with respect to a topological measure $\nu$, given on solid sets as $ \nu(A) = h(\mu_1(A), \mu_2(A), \mu_3(A))$ for a particular Boolean function $h$;
see \cite[L. 38]{OrjanAlf:Boolean}.
\end{remark}

\begin{theorem} \label{equivar}
Suppose $X_i$ ($i=1,2$) is a q-space or a noncompact q-space.
Suppose $Y$ is locally compact,  $\{ T_i: Y \longrightarrow X_1 \}$ is a finite collection of proper continuous maps, 
and $f: X_1  \longrightarrow X_2$ is a proper solid variable.   
\begin{enumerate}[label=(\roman*),ref=(\roman*)]
\item
For the odd number of maps $T_i, i =1, \ldots, 2n-1$  the inverse of the sample median is equivariant under  $f$, i.e.
$ T_{(n)}^{-1} \circ f^{-1} = (f \circ T)_{(n)}^{-1}$, where $(f \circ T)_{(n)}$ is the nth order statistic of $ \{f \circ T_i\}_{i=1}^{2n-1}$.
\item
Let $P$ be a deficient topological measure on $Y$. 
Then the g.d.s.m. is equivariant under  $f$, i.e. $\mu_{\{f \circ T_i \} } = \mu_{ \{T_i \} }  \circ f^{-1}$.  
\end{enumerate}
\end{theorem}

\begin{proof}
\begin{enumerate}[label=(\roman*),ref=(\roman*)]
\item
By Theorem \ref{ITfinva} $q_{\{T_i\}} \circ f^{-1} = q_{\{f \circ T_i\}}$, 
so by Theorem \ref{SaMeProbD}   $ T_{(n)}^{-1} \circ f^{-1} = (f \circ T)_{(n)}^{-1}$. 
\item
For the odd number of maps $T_i$ by Theorem \ref{SMedExt},
$\mu_{ \{f \circ T_i \} } = P \circ q_{\{f \circ T_i\}} = P(q_{\{T_i\}} \circ f^{-1}) = \mu_{\{T_i\}} \circ f^{-1}$. 
For the even number of maps $T_i$ the statement follows by Definition \ref{SampleMedE}.
\end{enumerate}
\end{proof} 

By Theorem \ref{equivar}, the inverse of an nth order statistic and the g.d.s.m. are equivariant under a variety of transformations, such as the ones in 
Example \ref{solidVar}.

\begin{remark}
Theorem \ref{SMedExt},  Theorem \ref{ITfinva}, Theorem \ref{SaMeProbD}\ref{oddSaMe},\ref{evenSaMe}, 
and Theorem \ref{equivar} when $X$ is a metric q-space
(and with a slightly different definition of an image transformation)  
can be found in at least two of \cite{AlfMultidimMedian}, \cite{AlfImTrans}, \cite{AlfMedian}. 
In our definitions and results we don't require any of the spaces to be metric.   
Part \ref{evenSaMeAlf} of Theorem \ref{SaMeProbD} is the definition of the g.d.s.m. for an even-numbered collection of maps in
compact spaces, see 
\cite[Def. 4]{AlfMultidimMedian} and \cite[Def. 20]{AlfMedian}.
\end{remark}

In this paper we showed how (d-) image transformation, quasi-homomorhisms and their adjoints can be used to create some 
interesting generalizations in the theory of fractals and probability $\&$ statistics. 
Image transformations also give an interesting model for a sponge-like structure on a compact space where small holes "don't matter", see 
 \cite[Ex. 5.16]{Pedersen}. 
To study other uses of quasi-linear maps and (d-) image transformations would be an interesting subject for another paper.  


$\\$
{\bf{Acknowledgments}}:
The author would like to thank the Department of Mathematics at the University of California Santa Barbara for its supportive environment.


\end{document}